\documentclass[twoside,11pt]{article}

\usepackage{blindtext}

\usepackage[dvipsnames]{xcolor}
\usepackage{subfig}
\usepackage{mwe}
\usepackage{amsmath}

\usepackage{jmlr2e}

\newcommand{\norm}[1]{\left\lVert#1\right\rVert}
\newcommand{\innerpro}[1]{\left\langle#1\right\rangle}
\newcommand{\bX}{\mathbf{X}}
\newcommand{\lam}{\lambda}

\newcommand{\sxp}{S_{X_p}}
\newcommand{\sxq}{S_{X_q}}
\newcommand{\hk}{{\mathcal{H}_K}}
\newcommand{\ie}{i.e. }

\DeclareSymbolFont{bbm}{U}{msb}{m}{n}
\DeclareSymbolFontAlphabet{\mathbbm}{bbm}

\def\R{\mathbbm R}
\def\N{\mathbbm N}

\newtheorem{assumption}[theorem]{Assumption}

\usepackage{lastpage}
\jmlrheading{xx}{2023}{0-\pageref{LastPage}}{5/23}{x/xx}{00-0000}{Duc Hoan Nguyen, Werner Zellinger and Sergei Pereverzyev}
\ShortHeadings{On regularized Radon-Nikodym differentiation}{Nguyen, Zellinger and Pereverzyev}
\firstpageno{1}

\begin{document}

\title{On regularized Radon-Nikodym differentiation}

\author{\name Duc Hoan Nguyen \email duc.nguyen@ricam.oeaw.ac.at \\
        \name Werner Zellinger \email werner.zellinger@ricam.oeaw.ac.at \\
        \name Sergei Pereverzyev \email sergei.pereverzyev@oeaw.ac.at \\
       \addr Johann Radon Institute for Computational and Applied Mathematics\\
       Austrian Academy of Sciences\\
       Altenberger Straße 69, 4040 Linz, Austria
}

\editor{}

\maketitle

\begin{abstract}%   <- trailing '%' for backward compatibility of .sty file
%We discuss the problem of estimation of Radon-Nikodym derivatives. This problem appears in various applications, such as covariate shift adaptation (importance sampling), outlier detection (likelihood-ratio test), feature selection (mutual information), and conditional probability estimation. To address the above problem, we employ the general regularization scheme in reproducing kernel Hilbert spaces. The convergence rate of the corresponding regularized algorithm is established by taking into account both the smoothness of the derivative and the capacity of the space in which it is estimated. It is done in terms of general source conditions and the regularized Christoffel functions. The theoretical results are illustrated by numerical simulations. 
We discuss the problem of estimating Radon-Nikodym derivatives. This problem appears in various applications, such as covariate shift adaptation, likelihood-ratio testing, mutual information estimation, and conditional probability estimation. To address the above problem, we employ the general regularization scheme in reproducing kernel Hilbert spaces. The convergence rate of the corresponding regularized algorithm is established by taking into account both the smoothness of the derivative and the capacity of the space in which it is estimated. This is done in terms of general source conditions and the regularized Christoffel functions. We also find that the reconstruction of Radon-Nikodym derivatives at any particular point can be done with high order of accuracy. Our theoretical results are illustrated by numerical simulations. 
\end{abstract}

 
\section{Introduction}
\label{sec:intro}
This paper is focused on the use of regularized kernel methods in the context of estimating the ratio of two probability density functions, which can also be called the Radon-Nikodym derivative of the corresponding probability measures.

Recently the estimation of Radon-Nikodym derivatives has gained significant attention due to its potential applications in such tasks as covariate shift adaptation, outlier detection, divergence estimation, and conditional probability estimation. Here we may refer to \cite{sugiyama2012book} and references therein. In order to address the above problem, various kernel-based approaches are available. In particular, several regularization schemes in reproducing Kernel Hilbert space (RKHS) can be employed \cite{nguyen2010, kanamori2012, que2013, schuster2020,gizewski2022}.

As can be seen from the above studies, the convergence of algorithms for Radon-Nikodym differentiation is influenced not only by the smoothness of the approximated function but also by the capacity of the approximating space. Though there are several studies that employed a particular regularization technique, such as Tikhonov–Lavrentiev regularization, to the best of our knowledge there is no study considering more general regularization schemes and taking into account both the above-mentioned factors, \ie smoothness and capacity. For example, in \cite{kanamori2012} and \cite{que2013} (see Type I setting there) only the capacity of the approximating space has been incorporated into error estimations, and in \cite{gizewski2022} and \cite{schuster2020} only the smoothness has been considered. 

Besides, since in some applications the point values of the Radon-Nikodym derivatives are of interest, it seems natural to study their approximation in spaces, where pointwise evaluations are well-defined. However, in \cite{kanamori2012} and \cite{que2013} the approximation has been analyzed in the space of integrable functions, where pointwise evaluations are not well-defined. 

In the present paper, we aim to overcome the above limitations.  More precisely, we study general regularization schemes and analyze their accuracy with respect to both the smoothness of the Radon-Nikodym derivative and the capacity of the RKHS in which it is estimated. This is done in terms of general source conditions and regularized Christoffel functions. We then establish accuracy bounds of the corresponding regularized algorithm in the norm of RKHS and pointwise.  Finally, we present some  numerical  illustrations supporting our theoretical results.

\section{Assumptions and auxiliaries}\label{sec:assum}
%In this section we provide details about the set up which is considered here, and we consider the capacity of the space in terms of the regularized Christoffel functions.

%Assume that for all $x \in \mathbf{X}, p(x) > 0.$ 

In the problem of estimation of Radon-Nikodym derivatives, we consider two probability measures $p$ and $q$ on a space $\mathbf{X} \subset \R^d$.
The information about the measures is only provided in the form of samples $X_p = \{x_1, x_2,\ldots,x_n\} $ and $X_q = \{x_1', x_2',\ldots,x_m'\}$ drawn independently and identically (i.i.d) from $p$ and $q$ respectively. Moreover, we assume that there is a function $\beta: \mathbf{X} \rightarrow [0,\infty) $, 
%such that 
%\begin{align*}
%	d q (x) = \beta (x)  d p (x).
%\end{align*}
%Then $\beta (x)$ 
which can be viewed as the Radon-Nikodym derivative $\frac{dq}{dp}$ of the probability measure $q(x)$ with respect to the probability measure $p(x)$, and for any measurable set $ A \subset \mathbf{X} $ it holds
$$\int_A dq (x) = \int_A \beta (x) d p (x).$$  
%Moreover, as in \cite{huang2006}, we assume that $\beta(x)$ is uniformly bounded on $\mathbf{X}$, such that $| \beta(x)| \leq b_0$ for some $b_0 > 0 $ and any $x \in \mathbf{X}$.

Our goal is to approximate the Radon-Nikodym derivative $\beta (x) = \frac{dq}{dp}$ by some function $ \hat{\beta} (x)$ based on the observed samples. As it has been already explained in Introduction, we in fact need a strategy that ensures a good pointwise approximation to the derivatives $\beta (x)$. Then it seems to be logical to estimate $\beta (x)$ in the norm of some RKHS, in which pointwise evaluations are well-defined. 

Let $\mathcal{H}_K$ be a reproducing Kernel Hilbert space with a positive-definite function $K: \bX \times \bX \rightarrow \R$ as reproducing kernel. We assume that $K$ is a continuous and bounded function, such that for any $x \in \mathbf{X} $ $$\norm{K(\cdot,x)}_{\mathcal{H}_K} = \innerpro{K(\cdot,x),K(\cdot,x)}_{\mathcal{H}_K}^\frac{1}{2} = [K(x,x)]^\frac{1}{2} \leq \kappa_0 < \infty. $$
Let $L_{2,\rho}$ be the space of square-integrable functions $f: \mathbf{X} \rightarrow \mathbb{R}$ with respect to the probability measure $\rho$. We define $J_q: \mathcal{H}_K \hookrightarrow L_{2, q} $ and $J_p: \mathcal{H}_K \hookrightarrow L_{2, p} $ as the inclusion operators, such that for instance, $ J_q $ assigns to a function $ g \in \mathcal{H}_K $ the same function seen as an element of $ L_{2,q} $. In the sequel, we distinguish two sample operators 
\begin{align*}
	&S_{X_q} f = (f(x_1'), f(x_2'),\ldots,f(x_m')) \in \R^m,\\
	&S_{X_p} f = (f(x_1), f(x_2),\ldots,f(x_n)) \in \R^n,
\end{align*}
acting from $\mathcal{H}_K$ to $\R^m$ and $\R^n$, where the norms in later spaces are generated by $m^{-1}$-times and $n^{-1}$-times the standard Euclidean inner products, such that, for example, for $ u = (u_1,u_2,\ldots,u_m), w = (w_1,w_2,\ldots,w_m) \in \R^m,$
$$ \innerpro{u,w}_{\R^m} = \frac{1}{m} \sum_{j = 1}^m u_j w_j, \hspace*{0.5cm} \norm{u}_{\R^m} = \innerpro{u,u}_{\R^m}^{\frac{1}{2}} = \left(   \frac{1}{m} \sum_{j = 1}^m u_j^2 \right)^{\frac{1}{2}}. $$
Then the adjoint operators $S_{X_q}^*: \R^m \rightarrow \mathcal{H}_K$ and 
$S_{X_p}^*: \R^n \rightarrow \mathcal{H}_K$ are given as
\begin{align*}
	&S_{X_q}^* u(\cdot) = \frac{1}{m} \sum_{j = 1}^m K(\cdot, x_j') u_j, \hspace*{0.5cm} u = (u_1,u_2,\ldots,u_m) \in \R^m, \\
	&S_{X_p}^* v(\cdot) = \frac{1}{n} \sum_{i = 1}^n K(\cdot, x_i) v_i,\hspace*{0.5cm} v = (v_1,v_2,\ldots,v_n) \in \R^n.
\end{align*}

In the literature, various RKHS-based approaches are available for a Radon-Nikodym derivative estimation. Here we may refer to \cite{kanamori2012} and to references therein. As it can be seen from \cite{que2013}, and also from \cite{gizewski2022}, conceptually, under the assumption that $\beta \in \hk$, several of the above approaches can be derived from a regularization of an operator equation, which can be written in our terms as 
\begin{align}\label{eq2.1}
	J_p^* J_p \beta = J_q^*J_q \mathbf{1}. 
\end{align}
Because of the compactness of the operator $J_p^* J_p$, its inverse $(J_p^* J_p)^{-1}$ cannot be a bounded operator in ${\mathcal{H}_K} $, which makes the equation (\ref{eq2.1}) ill-posed. Here, $\mathbf{1}$ is the constant function that takes the value $1$ everywhere, and almost without loss of generality, we assume that $\mathbf{1} \in {\mathcal{H}_K}$, because otherwise the kernel $K_1 (x, x') = 1 + K(x,x')$ will, for example, be used to generate a suitable RKHS containing all constant functions.

Since there is no direct access to the measures $p$ and $q$, the equation (\ref{eq2.1}) is inaccessible as well, but the samples $X_p$ and $X_q$ allow us to access its empirical version
\begin{align}\label{eq2.2}
	\sxp^* \sxp \beta = \sxq^* \sxq \mathbf{1}.
\end{align}
% where, similar to the above notations the operators $S_{N,S}^* : \R^N \rightarrow {\mathcal{H}_K}, $ $S_{M,T}^*: \R^M \rightarrow {\mathcal{H}_K}$ are given as 
%\begin{align*}
%	&S_{N,S}^* v(\cdot) = \frac{1}{N} \sum_{i = 1}^N K(\cdot, x_i) v_i,\hspace*{0.5cm} u = (v_1,v_2,\ldots,v_N) \in \R^N,\\
%	&S_{M,T}^* u(\cdot) = \frac{1}{M} \sum_{j = 1}^M K(\cdot, x_j') u_j, \hspace*{0.5cm} u = (u_1,u_2,\ldots,u_M) \in \R^M.
%\end{align*}
A regularization of equation (\ref{eq2.2}) may serve as a starting point for several approaches of estimating the Radon - Nikodym derivative $\beta$. For example, as it has been observed in \cite{kanamori2012, gizewski2022}, the known kernel mean matching (KMM) method \cite{huang2006} can be viewed as the regularization of (\ref{eq2.1}) by the method of quasi (least-squares) solutions, originally proposed by Valentin Ivanov (1963) and also known as Ivanov regularization (see, for example, \cite{oneto2016} and \cite{page2019} for its use in the context of learning). At the same time, from Theorem 1 of \cite{kanamori2012} it follows that the kernelized unconstrained least-squares importance fitting (KuLSIF) proposed in \cite{kanamori2012} is in fact the application of the Lavrentiev regularization scheme to the empirical version (\ref{eq2.2}) of the equation (\ref{eq2.1}), that is in KuLSIF we have 
\begin{align} \label{eq:beta_tilde}
	\hat{\beta} = \beta_{\mathbf{X}}^\lambda =  (\lam I + \sxp^* \sxp)^{-1} \sxq^*\sxq \mathbf{1}.
\end{align}

As we have already mentioned in Introduction, early bounds of the accuracy of Radon-Nikodym numerical differentiation have relied only on the capacity of the approximating space. For example, in \cite{nguyen2010, kanamori2012} the capacity of the underlying space $\hk$ has been measured in terms of the so-called bracketing entropy, and in \cite{kanamori2012} the value of the regularization parameter $\lam$ in KuLSIF (\ref{eq2.2}) has been chosen depending on that capacity measure. Note that, in such approach, there is no possibility of incorporating into the regularization the information about other factors, such as the smoothness of the approximated derivative $\beta$, which, as we know from \cite{gizewski2022}, also influences the regularization accuracy.
Therefore, in the present study, we follow \cite{pauwels2018} and employ the concept of the so-called regularized Christoffel function that allows direct incorporation of the regularization parameter $\lam$ into the definition of a capacity measure. Consider the function 
\begin{align}\label{def:Chris_func}
	C_\lam (x) = \innerpro{K(\cdot,x), (\lambda I +J_{p}^* J_{p} )^{-1} K(\cdot,x) }_{\mathcal{H}_K} = \norm{(\lambda I + J_p^* J_p)^{-\frac{1}{2}} K(\cdot,x) }^2_{\mathcal{H}_K}
\end{align}
Note that in \cite{pauwels2018} the reciprocal of $C_\lam (x)$, \ie $\frac{1}{C_\lam (x)}$, was called the regularized Christoffel function, but for the sake of simplicity, we will keep the same name also for (\ref{def:Chris_func}). Note also that in the context of supervised learning where usually only one probability measure, say $p$, is involved, the expected value
\[
\mathcal{N} (\lam) = \int_\bX C_\lam (x) d p(x)
\]
of $C_\lam (x)$, called the effective dimension, has been proven to be useful \cite{caponnetto2007}. This function is used as a capacity measure of $\hk$.

At the same time, if more than one measure appears in the supervised learning context, as is for example the case in the analysis of Nystr{\"o}m subsampling \cite{rudi2015, shuai2019Analysis}, then the $C$-norm of the regularized Christoffel function 
\begin{align}\label{def:N_inf}
	\mathcal{N}_\infty (\lambda) := \sup_{x \in \mathbf{X}} C_\lambda (x) 
\end{align}
is used in parallel with the effective dimension $\mathcal{N} (\lam)$. This gives a hint that $\mathcal{N}_\infty (\lambda)$ could also be a suitable capacity measure for analysing the accuracy of Radon-Nikodym numerical differentiation.

%We make an additional assumption, as in \cite{huang2006}, we assume that $\beta(x)$ is uniformly bounded on $\mathbf{X}$, such that $| \beta(x)| \leq b_0$ for some $b_0 > 0 $ and any $x \in \mathbf{X}$.

%Moreover, in the sequel we adopt the convention that $c$ denotes a generic positive coefficient, which can vary from appearance to appearance and may only depend on basic parameter such as $p$, $q$, $\kappa_0$, $b_0$ and others introduced below.

We will need the following statement.

\begin{lemma}\label{lem1}
	Let $b_0 > 0$ be such that $| \beta(x)| \leq b_0$ for any $x \in \bX$. Then with probability at least $1 - \delta$ we have
\begin{align*}
    \norm{(\lambda I + J_p^* J_p)^{-\frac{1}{2}} (\sxp^* \sxp \beta - \sxq^* \sxq \mathbf{1}) }_{\mathcal{H}_K} \leq \left( 1 + \sqrt{2 \log \frac{2}{\delta}}  \right) \sqrt{\mathcal{N}_\infty (\lambda)}  \sqrt{\frac{b_0}{n} + \frac{1}{m}}.
\end{align*}	
\end{lemma}
The proof of Lemma \ref{lem1} is based on Lemma 4 of \cite{huang2006}, which we formulate in our notations as follows

\begin{lemma} \label{lemma4_huang06} (\cite{huang2006})
	Let $\phi$ be a map from $\mathbf{X}$ into $\mathcal{H}_K$ such that $\norm{\phi(x)}_{\mathcal{H}_K} \leq R$ for all $x \in \mathbf{X}$. Then with probability at least $1 - \delta$ it holds 
	$$\norm{ \frac{1}{m} \sum_{j = 1}^m \phi (x_j') -  \frac{1}{n} \sum_{i = 1}^n \beta (x_i) \phi (x_i)  }_{\mathcal{H}_K} \leq \left( 1 + \sqrt{2 \log \frac{2}{\delta}} \right) R \sqrt{ \frac{b_0^2}{n} + \frac{1}{m}}.$$
\end{lemma}
Now we can return to Lemma \ref{lem1}.
\begin{proof} 
	We define a map $\phi : \mathbf{X} \rightarrow \mathcal{H}_K$ as $ \phi (x) = (\lambda I + J_p^* J_p)^{-\frac{1}{2}} K(\cdot,x) $, $x \in \mathbf{X}$. It is clear that 
	$$\norm{\phi (x)}_{\mathcal{H}_K} = \norm{(\lambda I + J_p^* J_p)^{-\frac{1}{2}} K(\cdot,x) }_{\mathcal{H}_K} = \sqrt{C_\lam (x)} \leq \sqrt{\mathcal{N}_\infty (\lambda)}.$$
	Therefore, for the map $\phi $ the condition of the above Lemma \ref{lemma4_huang06} is satisfied with $R = \mathcal{N}_\infty (\lambda)$. Then directly from that lemma, we have 
	\begin{align*}
		\norm{(\lambda I + J_p^* J_p)^{-\frac{1}{2}} (\sxp^* \sxp \beta - \sxq^* \sxq \mathbf{1}) }_{\mathcal{H}_K}  &= \norm{ \frac{1}{n} \sum_{i = 1}^n \beta(x_i) \phi (x_i) - \frac{1}{m} \sum_{j = 1}^m \phi (x_j') }_{\mathcal{H}_K} \\
		%			&= \norm{\frac{1}{m} \sum_{j = 1}^m \phi (x_j') -   \frac{1}{n} \sum_{i = 1}^n \beta(x_i) \phi (x_i)}_{\mathcal{H}_K} \\
		& \leq \left( 1 + \sqrt{2 \log \frac{2}{\delta}} \right) \left( \sqrt{\frac{b_0^2}{n} + \frac{1}{m}} \right) \sqrt{\mathcal{N}_\infty (\lambda)}.
%		& \leq c \left( \log^\frac{1}{2} \frac{1}{\delta}\right) \left( m^{- \frac{1}{2}} + n^{- \frac{1}{2}} \right) \sqrt{\mathcal{N}_\infty (\lambda)},
	\end{align*}
%	that proves the statement of Lemma \ref{lem1}.
\end{proof}

\section{General regularization scheme and general source conditions}

All of the available regularization methods have the potential to be employed for the regularization of equation (\ref{eq2.2}). In particular, we will use a general regularization scheme to construct a family of approximate solutions $\beta_{\mathbf{X}}^\lam$ of (\ref{eq2.1}) as follows
\begin{align}\label{eq:beta_g_lam}
	\beta_{\mathbf{X}}^\lam =  g_\lam (\sxp^* \sxp) \sxq^*\sxq \mathbf{1},
\end{align}
where $\{g_\lam\}$ is a family of operator functions parametrized by a regularization parameter $\lam >0$.
\subsection{General regularization scheme}\label{sec:gen_reg_scheme}
Recall (see, for example, Definition 2.2 in \cite{1st_book}) that regularization schemes can be indexed by parametrized functions $g_\lambda : [0,c] \rightarrow \R$,  $\lambda >0.$ The only requirements are that there are positive constants $\gamma_0, \gamma_{- \frac{1}{2}}, \gamma_{-1}$ for which 
\begin{align}\label{eq:gen_reg3.1.1}
	\begin{split}
		\sup_{0 < t \leq c} |1 - t g_\lambda (t)| &\leq \gamma_0, \\
		\sup_{0 < t \leq c} \sqrt{t} |g_\lambda (t)|&\leq \frac{\gamma_{- \frac{1}{2}}}{ \sqrt{\lambda}}, \\
		\sup_{0 < t \leq c}  |g_\lambda (t)| &< \frac{\gamma_{-1}}{\lambda}.
	\end{split}
\end{align} 
Here and in the sequel, we adopt the convention that $c$ denotes a generic positive coefficient, which can vary from appearance to appearance and may only depend on basic parameters such as $p$, $q$, $\kappa_0$, $b_0$, and others introduced below.

The qualification of a regularization scheme indexed by $g_\lambda$ is the maximal $s > 0 $ for which 
\begin{align}\label{eq:gen_reg3.1.2}
	\sup_{0 < t \leq c} t^s | 1 - t g_\lambda (t)  | \leq \gamma_s \lambda^s, 
\end{align}
where $\gamma_s$ does not depend on $\lambda$. Following Definition 2.3 of \cite{1st_book} we also say that the qualification $s$ covers a non-decreasing function $\varphi : [0, c ] \rightarrow \R$, $\varphi (0) = 0,$ if the function $t \rightarrow \frac{t^s}{\varphi (t)}$ is non-decreasing for $t \in (0,c].$ Note that the higher qualification is the more rapidly increasing functions can be covered, and in this way, as it can be seen below, the more smoothness of approximated solutions can be utilized in the regularization.

Observe that the Lavrentiev regularization used in KuLSIF (\ref{eq:beta_tilde}) is indexed by $g_\lam (t) = (\lam + t)^{-1}$ and has qualification $s = 1$. The qualification of this regularization scheme can be increased if one employs the so-called iteration idea, according to which regularized algorithms need to be repeated such that, for example, the approximate Radon-Nikodym derivative $\beta_\mathbf{X}^\lambda = \beta_{\mathbf{X}}^{\lambda,l}$ obtained in the previous $l$-th step plays the role of an initial guess for the next approximation $\beta_\mathbf{X}^\lambda = \beta_{\mathbf{X}}^{\lambda,l+1}$. In this regularization the approximation (\ref{eq:beta_g_lam}) can be obtain iteratively as follows
\begin{align*}
	\beta_{\mathbf{X}}^{\lam,0} &=0,\\
	\beta_{\mathbf{X}}^{\lam,l} &= (\lam I + \sxp^* \sxp)^{-1} (\sxq^*\sxq \mathbf{1} + \lam \beta_{\mathbf{X}}^{\lam,l-1}), \hspace*{0.3cm}l \in \N.
\end{align*}
After $k$ such iterations we obtained the approximation $\beta_\mathbf{X}^\lambda = \beta_{\mathbf{X}}^{\lambda,k}$ that can be represented in the form (\ref{eq:beta_g_lam}) with 
$$g_\lambda (t) = g_{\lambda,k} (t) = \frac{1 - \frac{\lambda^k}{(\lambda + t)^k}}{t}.$$
The regularization indexed by $g_{\lambda,k} (t)$ has the qualification $k$ that can be taken as large as desired. Moreover, for $g_\lambda (t) = g_{\lambda,k} (t)$ the requirements (\ref{eq:gen_reg3.1.1}), (\ref{eq:gen_reg3.1.2}) are satisfied with $\gamma_0 = 1, \gamma_{- \frac{1}{2}} = k^{\frac{1}{2}}, \gamma_{-1} = k, \gamma_k = 1.$

For the sake of shortness, we introduce the residual function
$$r_\lam (t) := 1 - t g_\lam(t), $$
for which (\ref{eq:gen_reg3.1.1}), (\ref{eq:gen_reg3.1.2}) give the bounds $r_\lam (t) \leq \gamma_0$ and $ |t^s r_\lam (t) | \leq \gamma_s \lambda^s.$

\subsection{General source conditions}
%To further establish error estimates we need to restrict our objective function $\beta$ to a compact set with certain regularity properties, often called smoothness. The smoothness is then given in terms of a general source condition, generated by some index functions.

As mentioned in the previous section, the equation (\ref{eq2.1}) is inaccessible, but the result \cite{mathe2008} of the regularization theory tells us that there is always a continuous, strictly increasing function $ \varphi : [0,\|J_p^* J_p \|_{\mathcal{H}_K} ] \rightarrow \R$ that obeys $\varphi (0) = 0 $ and allows the representation of the solution of (\ref{eq2.1}) in terms of the so-called source condition:
\begin{align}\label{eq:smooth3.2.1}% label{eq2.1.8}
	\beta = \varphi (J_p^* J_p ) \nu_q, \hspace*{0.3cm} \nu_q \in {\mathcal{H}_K}.
\end{align}
The function $\varphi$ above is usually called the index function. Moreover, for every $\epsilon > 0$ one can find such $\varphi$ that (\ref{eq:smooth3.2.1}) holds true for $\nu_q$ with $$\|\nu_q\|_{\mathcal{H}_K} \leq (1 + \epsilon) \|\beta\|_{\mathcal{H}_K}. $$
Note that since the operator $J_p^* J_p$ is not accessible, there is a reason to restrict ourselves to consideration of such index functions $\varphi$, which allow us to control perturbations in the operators involved in the definition of source conditions. A class of such index functions has been discussed in \cite{mathe2003, bauer2007}, and here we follow those studies. Namely, we consider the class $\mathcal{F} = \mathcal{F} (0,c)$ of index functions $\varphi : [0,c] \rightarrow \R_+$ allowing splitting $\varphi (t) = \vartheta (t) \psi (t)$ into monotone Lipschitz part $\vartheta, \vartheta (t) = 0,$ with the Lipschitz constant equal to 1, and an operator monotone part $\psi, \psi(0)= 0.$ 

Recall that a function $\psi$ is operator monotone on $[0,c]$ if for any pair of self-adjoint operators $U, V$ with spectra in $[0,c]$ such that $U \leq V$ (\ie $V - U$ is an non-negative operator) we have $\psi (U) \leq \psi (V)$.

Examples of operator monotone index functions are $\psi (t) = t^\nu,$ $\psi (t) = \log^{-\nu} \left( \frac{1}{t} \right),$ $\psi (t) = \log^{-\nu} \left( \log \frac{1}{t} \right), 0 < \nu  \leq 1,$ while an example of a function $\varphi$ from the above defined class $\mathcal{F}$ is $\varphi(t) = t^r \log^{-\nu} \left( \frac{1}{t} \right), r > 1, 0 < \nu \leq 1,$  since it can be splitted in a Lipschitz part $\vartheta (t) = t^r$ and an operator monotone part $\psi (t) = \log^{-\nu} \left( \frac{1}{t} \right).$
%it leads us to restrict the objective function $\beta$ to a compact set with certain regularity properties, often called smoothness. The smoothness is generated by some index functions and given in terms of a general source condition.

We will need the result of Proposition 3.1 in \cite{2nd_book}, which we formulate in
our notations as follows
\begin{lemma}\label{lem:prop3.1_sergei_book} (\cite{2nd_book})
	Let $\varphi : [0,c] \rightarrow \R$, $\varphi (0)=0,$ be any non-decreasing index function. If the qualification $s$ of the regularization indexed by a family $\{	g_\lambda   \}$ covers the function $\varphi$, then for any $ \lambda \in  (0,c] $ it holds 
	\[
	\sup_{t \in [0,c]} | r_\lam (t) \varphi (t)| \leq \gamma_{0,s} \varphi(\lambda),
	\]
	where $ \gamma_{0,s} = \max \{\gamma_0, \gamma_s \} $, and $ \gamma_0, \gamma_s $ are the coefficients appearing in (\ref{eq:gen_reg3.1.1}) and in (\ref{eq:gen_reg3.1.2}).  
\end{lemma}

To estimate the regularized Christoffel functions we slightly generalize a source condition for kernel sections $K(\cdot,x)$ that has been used in various contexts in \cite{shuai2019Analysis} and \cite{devito2014}.

\begin{assumption} \label{assum:source_cond_kernel}
	(Source condition for kernel) There is an operator concave index function $\xi : [0, \norm{J^*_p J_p}] \to [0,\infty)$ and $\xi^2$ is covered by qualification $s=1$ such that, for all $x \in \bX $,
	$$K(\cdot, x) = \xi (J^*_p J_p) v_x, \hspace*{0.3cm} \norm{v_x}_\hk \leq c, %\text{ with } c \text{ is a constant.} 
	$$
	where $c$ does not depend on $x$.	
\end{assumption}
We mention the following consequence of Assumption \ref{assum:source_cond_kernel}.
\begin{lemma}\label{lem:N_infty_bound}
	Under Assumption \ref{assum:source_cond_kernel},
	$$ \mathcal{N}_\infty (\lam) \leq c \frac{\xi^2(\lam)}{\lam}. $$
\end{lemma}
\begin{proof}
	This simply follows from
	\begin{align*}
		\mathcal{N}_\infty (\lambda) &= \sup_{x \in \mathbf{X}} \norm{(\lambda I + J_p^* J_p)^{-\frac{1}{2}} K(\cdot,x) }^2_{\mathcal{H}_K} \\
		&= \sup_{x \in \mathbf{X}} \norm{(\lambda I + J_p^* J_p)^{-\frac{1}{2}} \xi (J^*_p J_p) v_x }^2_{\mathcal{H}_K} \\
		&\leq \sup_{x \in \mathbf{X}} \norm{v_x}^2_\hk \sup_{t \in \left[ 0,||J^*_pJ_p|| \right] }   | (\lam + t)^{-\frac{1}{2}} \xi (t) |^2 \\		
		&\leq c \sup_t  | (\lam + t)^{-1} \xi^2 (t) | \\
		&\leq c \lambda^{-1} \sup_t  \bigg| \left(1 - t(\lam + t)^{-1} \right) \xi^2 (t) \bigg| \\
		& \leq c \lambda^{-1} \sup_t  | r_\lam (t) \xi^2 (t) |\\
		&\leq c \frac{\xi^2(\lam)}{\lam},
	\end{align*}
	where in the last inequality we have used Lemma \ref{lem:prop3.1_sergei_book}, Assumption \ref{assum:source_cond_kernel}, and the fact that the Lavrentiev regularization indexed by $g_\lam (t) = (\lam +t)^{-1}$ has the qualification $s = 1$.
\end{proof}
\begin{remark}
	In \cite{pauwels2018}, the asymptotic behavior of the regularized Christoffel functions $C_\lam (x)$ as $\lam \to 0$ has been analyzed for translation invariant kernels $K(x,t) = K(x-t)$. Our Lemma \ref{lem:N_infty_bound} can be viewed as an extension of that analysis based on the general source conditions on the kernel sections $K_x(\cdot) = K(\cdot,x) \in \hk$.
\end{remark}
\section{Error estimates in RKHS}
In this section, we discuss error estimates between $\beta$ and $\beta_{\mathbf{X}}^\lam$ for RKHS norm. To this end, we consider an auxiliary regularized approximation $\bar{\beta}^\lam$ defined as follows
\begin{align}\label{eq:beta_bar}
	\bar{\beta}^\lam = g_\lam (\sxp^*\sxp) \sxp^*\sxp \beta.
\end{align} 
Then we decompose the error bound into two parts:
\begin{align}\label{eq:decom_bound}
	\norm{\beta - \beta^\lam_{\mathbf{X}}}_\hk \leq \norm{\beta - \bar{\beta}^\lambda }_\hk + \norm{\bar{\beta}^\lambda - \beta^\lam_{\mathbf{X}}}_\hk.
\end{align}
We call the first term on the right-hand side of (\ref{eq:decom_bound}) the approximation error, and the second term the noise propagation error.

Following \cite{lu2020}, we introduce the functions
\begin{align}
	\mathcal{B}_{n,\lam} &:= \frac{2 \kappa_0}{\sqrt{n}} \left( \frac{\kappa_0}{\sqrt{n \lam}} + \sqrt{\mathcal{N}(\lam)}\right),\label{def:bn} \\
	\Upsilon (\lam) &:= \left( \frac{\mathcal{B}_{n,\lam}}{\sqrt{\lam}}   \right)^2 + 1, \label{def:Upsilon_lam}
\end{align}
which will be useful in the subsequent analysis.

 Moreover, we need the following estimates from \cite{lu2020} that are valid with probability at least $1 - \delta$ and can be written in our notations as
\begin{align}
	\norm{J^*_p J_p - \sxp^* \sxp}_{\hk \to \hk} &\leq \frac{4 \kappa_0^2}{\sqrt{n}} \log \frac{2}{\delta}, \label{1st_bound_of_3}\\
	\norm{(\lam I + J^*_p J_p )^{-1/2} (J^*_p J_p - \sxp^* \sxp)}_{\hk \to \hk} &\leq \mathcal{B}_{n,\lam} \log \frac{2}{\delta},\label{2nd_bound_of_3} \\
	\Xi := \norm{(\lam I + J_p^* J_p) (\lam I + \sxp^* \sxp)^{-1}}_{\hk \to \hk}  & \leq 2 \left[ \left( \frac{\mathcal{B}_{n,\lam} \log \frac{2}{\delta}}{\sqrt{\lam}}   \right)^2 + 1  \right]. \label{def:Xi}
\end{align}

%\begin{align}
%		\norm{(\lam I + J_p^* J_p)^{1/2} (\lam I + \sxp^* \sxp)^{-1/2}}_{\hk \to \hk} &\leq \Xi^{1/2}.
%\end{align}

\begin{proposition}\label{prop:appro_bound}
	(Approximation error bound).

	\begin{enumerate}
		\item If $\beta$ meets source condition (\ref{eq:smooth3.2.1}), where  $\varphi$ is an operator monotone index function, then
		$$ \norm{\beta - \bar{\beta}^\lam}_\hk \leq c (\gamma_0 + \gamma_{-1}) \Xi \varphi (\lam), \hspace*{0.3cm} 0 < \lam \leq \kappa_0.  $$
		\item If $\beta$ meets source condition (\ref{eq:smooth3.2.1}), where $\varphi = \vartheta \psi \in \mathcal{F} (0,c)$ with $c$ is large enough and if the qualification of the regularization $g_\lam$ covers $\vartheta (t) t^\frac32$, then
		$$ \norm{\beta - \bar{\beta}^\lam}_\hk \leq c \Xi \varphi (\lam) + \psi(\kappa_0) \gamma_0 \norm{J_p^* J_p - \sxp^* \sxp}_{\hk \to \hk}. $$
	\end{enumerate}
\end{proposition}
\begin{proof}
	The Proposition \ref{prop:appro_bound} can be proved by repeating line by line the argument of the proof of Proposition 4.3 in \cite{lu2020}, where the items denoted there as $T$, $T_x$, $f^\dagger$ and $\bar{f}_x^\lam$ should be substituted by $J_p^* J_p$, $\sxp^* \sxp$, $\beta$, and $\bar{\beta}^\lam$, respectively.
\end{proof}

\begin{proposition}\label{prop:propa_bound} (Noise propagation error bound). 
	Let $ \beta_{\mathbf{X}}^\lam, \bar{\beta}^\lam$ be defined by (\ref{eq:beta_tilde}), (\ref{eq:beta_bar}). Then with probability at least $1 - \delta$ it holds
	$$\norm{\bar{\beta}^\lam - \beta_{\mathbf{X}}^\lam}_\hk \leq c \left( \gamma_{- \frac{1}{2}}^2 + \gamma_{-1}^2 \right)^{\frac12} \Xi^{\frac12} \frac{1}{\sqrt{\lam}} \left( m^{- \frac{1}{2}} + n^{- \frac{1}{2}} \right) \sqrt{\mathcal{N}_\infty (\lambda)}  \left( \log^\frac{1}{2} \frac{1}{\delta}\right) .$$
\end{proposition}

\begin{proof}
	From (\ref{def:Xi}) and well-known Cordes inequality we have
	\begin{align}\label{eq:1/2Xi}
		\norm{(\lam I + J_p^* J_p)^{1/2} (\lam I + \sxp^* \sxp)^{-1/2}}_{\hk \to \hk} &\leq \Xi^{1/2}.
	\end{align}
	 Then using (\ref{eq:gen_reg3.1.1}) and Lemma \ref{lem1}, we can continue
	\begin{align*}
		\norm{\bar{\beta}^\lam - \beta_{\mathbf{X}}^\lam}_\hk \leq&  \norm{g_\lam (\sxp^* \sxp) (\sxq^* \sxq \mathbf{1} - \sxp^* \sxp \beta )}_\hk \\
		\leq & \left\lVert g_\lam (\sxp^* \sxp) (\lam I + \sxp^* \sxp)^\frac12 \right\lVert_\hk  \left\lVert (\lam I + \sxp^* \sxp)^{-\frac{1}{2}} (\lam I + J^*_p J_p)^\frac12 \right\lVert_\hk \times \\
		&\times \left\lVert (\lam I + J^*_p J_p)^{-\frac12} (\sxq^* \sxq \mathbf{1} - \sxp^* \sxp \beta )
		\right\lVert_\hk\\
		\leq &\sup_{0 < t \leq c} |g_\lam (t) (\lam +t)^\frac12| \Xi^{\frac12}   \norm{(\lambda I + J_p^* J_p)^{-\frac{1}{2}} (S_{X_p}^* S_{X_p} \beta - S_{X_T}^* S_{X_T} \mathbf{1}) }_{\mathcal{H}_K} \\
		\leq &(\gamma_{- \frac{1}{2}}^2 + \gamma_{-1}^2 )^{\frac12} \frac{1}{\sqrt{\lam}}  \Xi^{\frac12} \left( 1 + \sqrt{2 \log \frac{2}{\delta}} \right) \left( \sqrt{\frac{b_0^2}{n} + \frac{1}{m}} \right) \sqrt{\mathcal{N}_\infty (\lambda)} \\
		\leq & c \left( \gamma_{- \frac{1}{2}}^2 + \gamma_{-1}^2 \right)^{\frac12} \frac{1}{\sqrt{\lam}} \Xi^{\frac12} \left( m^{- \frac{1}{2}} + n^{- \frac{1}{2}} \right) \sqrt{\mathcal{N}_\infty (\lambda)}  \left( \log^\frac{1}{2} \frac{1}{\delta}\right) .
	\end{align*}
\end{proof}

The next proposition summaries of Proposition \ref{prop:appro_bound} and \ref{prop:propa_bound}.
%We summarize the error estimates in Proposition \ref{prop:appro_bound} and \ref{prop:propa_bound} in terms of the above functions $\mathcal{B}_{n,\lam},  $ $\Upsilon(\lam) $ from (\ref{def:bn}) - (\ref{def:Xi})  and $ \mathcal{N}_\infty (\lam)$ from (\ref{def:N_inf}).

\begin{proposition}\label{prop:combine_app_propa_bound}
	If $\beta$ meets source condition (\ref{eq:smooth3.2.1}), where  $\varphi$ is an operator monotone index function, then with probability at least $1 - \delta$ it holds
	\begin{align*}
		\norm{\beta - \beta^\lam_{\mathbf{X}}}_\hk \leq c \left( \Upsilon(\lam)\varphi(\lam) + \frac{1}{\sqrt{\lam}} [\Upsilon (\lam) ] ^\frac12 (m^{-\frac12} + n^{-\frac12}) \sqrt{\mathcal{N}_\infty(\lam)} \right) \left( \log \frac{2}{\delta}\right)^2.
	\end{align*}
	
	If $\beta$ meets source condition (\ref{eq:smooth3.2.1}), where $\varphi = \vartheta \psi \in \mathcal{F} (0,c)$ with $c$ is large enough, and if the qualification of the regularization $g_\lam$ covers $\vartheta (t) t^\frac32$ then with probability at least $1 - \delta$ the total error allows for the bound
	\begin{align*}
		\norm{\beta - \beta^\lam_{\mathbf{X}}}_\hk \leq c \left( \Upsilon(\lam)\varphi(\lam) + n^{-\frac12}+ \frac{1}{\sqrt{\lam}} [\Upsilon (\lam) ] ^\frac12 (m^{-\frac12} + n^{-\frac12}) \sqrt{\mathcal{N}_\infty(\lam)} \right) \left( \log \frac{2}{\delta}\right)^2.
	\end{align*}
\end{proposition}
\begin{proof}
	We first prove the results for $\beta$ meets source condition (\ref{eq:smooth3.2.1}), where  $\varphi$ is an operator monotone index function. Using the error estimates in Proposition \ref{prop:appro_bound} and \ref{prop:propa_bound}, and (\ref{def:bn}) - (\ref{def:Xi}), we have 
	\begin{align*}
		\norm{\beta - \beta^\lam_{\mathbf{X}}}_\hk \leq &c (\gamma_0 + \gamma_{-1}) \left[ \left( \frac{\mathcal{B}_{n,\lam} \log \frac{2}{\delta}}{\sqrt{\lam}}   \right)^2 + 1  \right] \varphi (\lam) \\
		& + \left( \gamma_{- \frac{1}{2}}^2 + \gamma_{-1}^2 \right)^{1/2} \frac{1}{\sqrt{\lam}}   \sqrt{ \left( \frac{\mathcal{B}_{n,\lam} \log \frac{2}{\delta}}{\sqrt{\lam}}   \right)^2 + 1  } \left( \log^\frac{1}{2} \frac{1}{\delta}\right) \left( m^{- \frac{1}{2}} + n^{- \frac{1}{2}} \right) \sqrt{\mathcal{N}_\infty (\lambda)} \\
		\leq & c \left( \Upsilon(\lam)\varphi(\lam) + \frac{1}{\sqrt{\lam}} [\Upsilon (\lam) ] ^\frac12 (m^{-\frac12} + n^{-\frac12}) \sqrt{\mathcal{N}_\infty(\lam)} \right) \left( \log \frac{2}{\delta}\right)^2.
	\end{align*}
	Similarly, if $\beta$ meets source condition (\ref{eq:smooth3.2.1}), with $\varphi = \vartheta \psi \in \mathcal{F} (0,c)$ and the qualification of the regularization $g_\lam$ covers $\vartheta (t) t^\frac32$, we have 
	\begin{align*}
		\norm{\beta - \beta^\lam_{\mathbf{X}}}_\hk \leq & c \left[ \left( \frac{\mathcal{B}_{n,\lam} \log \frac{2}{\delta}}{\sqrt{\lam}}   \right)^2 + 1  \right] \varphi (\lam) + \psi(\kappa_0) \gamma_0 \frac{4 \kappa_0^2}{\sqrt{n}} \log \frac{2}{\delta} \\
		&+ c \frac{1}{\sqrt{\lam}}   \sqrt{ \left( \frac{\mathcal{B}_{n,\lam} \log \frac{2}{\delta}}{\sqrt{\lam}}   \right)^2 + 1  } \left( \log^\frac{1}{2} \frac{1}{\delta}\right) \left( m^{- \frac{1}{2}} + n^{- \frac{1}{2}} \right) \sqrt{\mathcal{N}_\infty (\lambda)}  \\
		\leq & c \left( \Upsilon(\lam)\varphi(\lam) + n^{-\frac12}+ \frac{1}{\sqrt{\lam}} [\Upsilon (\lam) ] ^\frac12 (m^{-\frac12} + n^{-\frac12}) \sqrt{\mathcal{N}_\infty(\lam)} \right) \left( \log \frac{2}{\delta}\right)^2.
	\end{align*}
\end{proof}

We will also need the following statement proven in  \cite{lu2020} as Lemma 4.6.
\begin{lemma}\cite{lu2020}\label{lem4.6}
	There exists a $\lam_*$ satisfying $\mathcal{N}(\lam_*)/\lam_* = n$. For $\lam_* \leq \lam \leq \kappa_0$, there holds
	\begin{align}
		\mathcal{B}_{n,\lam} \leq \frac{2\kappa_0}{\sqrt{n}} \left( \sqrt{2} \kappa_0 + \sqrt{\mathcal{N} (\lam)} \right).
	\end{align}
	This yields 
	\begin{align}
		\Upsilon (\lam) \leq 1 + (4\kappa_0^2 + 2 \kappa_0)^2
	\end{align}
	and also
	\begin{align}
		\mathcal{B}_{n,\lam} \left( \mathcal{B}_{n,\lam} + \sqrt{\lam} \right) \leq (1+4\kappa_0)^4 \min \left\{ \lam, \sqrt{\frac{\kappa_0}{n}} \right\},
	\end{align}
for $n$ large enough.
\end{lemma}
%We recall the definition of $\lam^*$ from Lemma \ref{lem4.6}. For $\lam \geq \lam^*$ we can refine Proposition \ref{prop:combine_app_propa_bound} by using Lemma \ref{lem:N_infty_bound} and Lemma \ref{lem4.6}, i.e., replacing $N_\infty (\lam)$ and $\Upsilon (\lam)$ by the corresponding bounds.
For $\lam > \lam_*$ we can make the statement of Proposition \ref{prop:combine_app_propa_bound} more transparent.
\begin{theorem}\label{thm4.5}
	Let $K$ satisfies Assumption \ref{assum:source_cond_kernel}, and $\lam \geq \lam^*$. Then under the assumptions of Proposition \ref{prop:combine_app_propa_bound}, with probability at least $1 - \delta$, it holds
	\begin{align*}
		\norm{\beta - \beta^\lam_{\mathbf{X}}}_\hk \leq c \left( \varphi(\lam) +  (m^{-\frac12} + n^{-\frac12}) \frac{\xi (\lam)}{\lam} \right) \left( \log \frac{2}{\delta}\right)^2.
	\end{align*}
	Consider $\theta_{\varphi,\xi} (t) = \frac{\varphi(t) t}{\xi(t)}$ and $\lam = \lambda_{m,n} = \theta_{\varphi,\xi}^{-1} (m^{-\frac12} + n^{-\frac12})$, then
	$$\norm{\beta - \beta^\lam_{\mathbf{X}}}_\hk \leq c \varphi\left(\theta_{\varphi,\xi}^{-1} (m^{-\frac12} + n^{-\frac12})\right) \log^2 \frac1\delta.$$
\end{theorem}

\begin{remark}\label{rem_sec4}
	As we already mentioned, the accuracy of the approximation (\ref{eq:beta_g_lam}) in RKHS has also been estimated in Theorem 2 of \cite{gizewski2022}. In our terms, the result of \cite{gizewski2022} can be written as follows:
	\begin{align}\label{bound_in_acha}
		\norm{\beta - \beta_{\mathbf{X}}^\lam}_\hk  \leq c \varphi \left(\theta_\varphi^{-1} (m^{-\frac12} + n^{-\frac12})  \right) \log \frac{1}{\delta}.
	\end{align}
	where $\theta_\varphi (t) = \varphi (t) t$.
     To simplify the comparison of Theorem \ref{thm4.5} and (\ref{bound_in_acha}), let us consider the case when $\beta$ meets the source condition (\ref{eq:smooth3.2.1}) with $\varphi (t) = t^\eta$. %, $\eta\geq\frac12$. 
     %If $\lam = \lam_{m,n} = \theta_\varphi^{-1} (m^{-\frac12} + n^{-\frac12}) $, 
    In this case the bound (\ref{bound_in_acha}) can be reduced to 
	\begin{align}\label{rate_in_acha}
		\norm{\beta - \beta_{\mathbf{X}}^\lam}_\hk  = O \left((m^{-\frac12} + n^{-\frac12})^\frac{\eta }{\eta +1}\right).
	\end{align}
	It is noteworthy that the error bound established in Theorem 2 of \cite{gizewski2022} does not take into consideration the capacity of $\hk$. Such an additional factor can be accounted in terms of Assumption \ref{assum:source_cond_kernel} . Assume that $K$ satisfies Assumption 3.2 with $\xi (t) = t^\varsigma, 0< \varsigma  \leq \frac12$, then for $\lam = \lam_{m,n} = \theta_{\varphi,\xi}^{-1} (m^{-\frac12} + n^{-\frac12})$, the bound in Theorem \ref{thm4.5} gives  
	\[
	\norm{\beta - \beta_{\mathbf{X}}^\lam}_\hk = O \left((m^{-\frac12} + n^{-\frac12})^\frac{\eta }{\eta +1 - \varsigma } \right),
	\]
%	It is easy to see that for $\varsigma > 0$ this 
 that is better than the order of accuracy given by (\ref{rate_in_acha}). Then one can conclude that the bound in Theorem \ref{thm4.5} obtained by our argument generalizes, specifies, and refines the results of \cite{gizewski2022}. 
	
Recall that the bounds in Theorem \ref{thm4.5} 
% is established with the sufficient condition for regularization parameter, such that
 are valid for $\lambda > \lambda_*$. Using Lemma \ref{lem:N_infty_bound} and \ref{lem4.6} one can prove that $\lam = \lam_{m,n}$ also satisfies the above inequality. The corresponding proof can be easily recovered from Figure \ref{fig_lam}.
\end{remark}
\begin{figure}[!ht]
    \centering
    \includegraphics[scale=0.35]{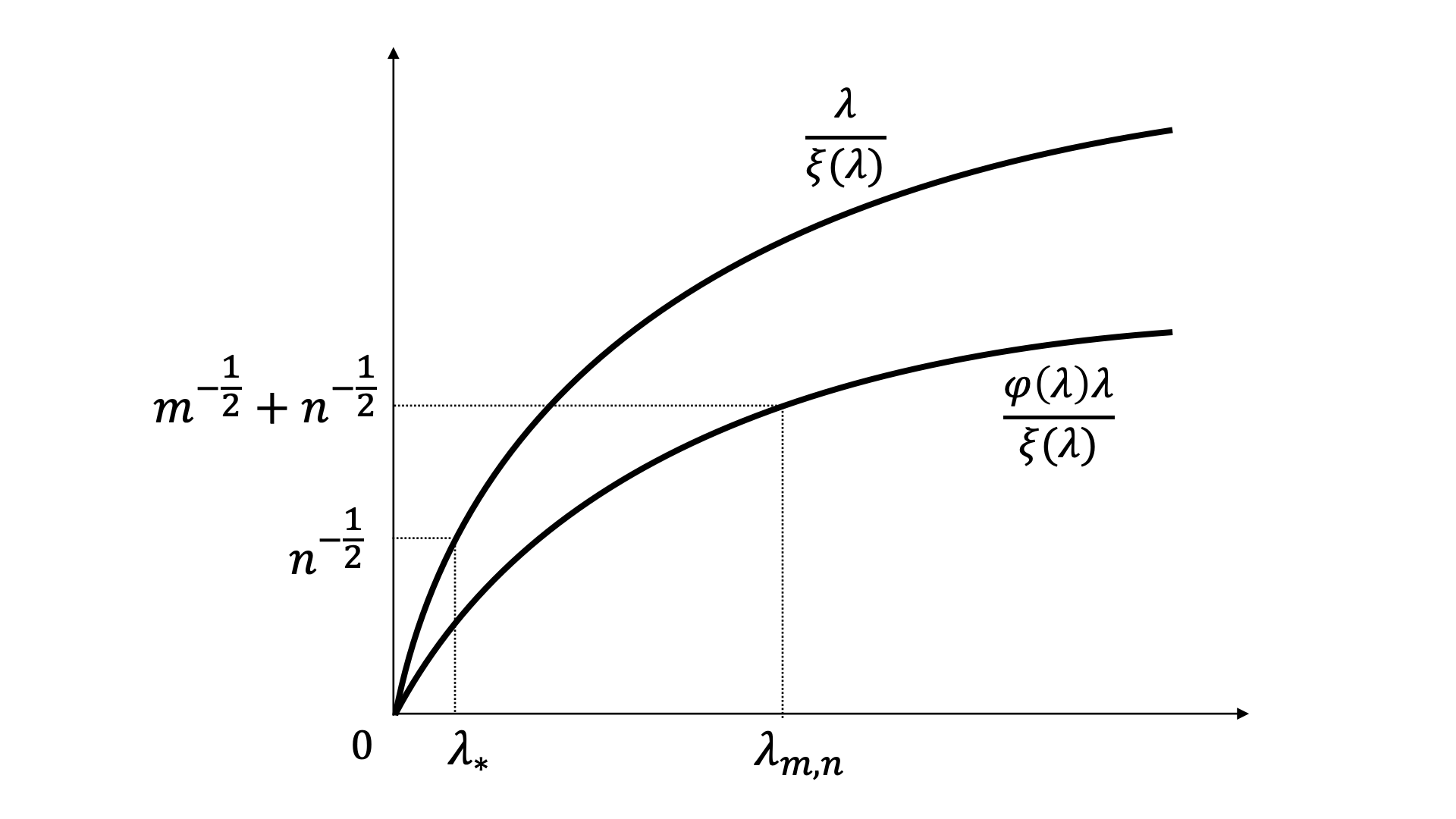}
    \caption{Relation between $\lam_*$ and $\lam_{m,n}$.}
    \label{fig_lam}
\end{figure}
\section{Error bounds for the pointwise evaluation}
In this section, we discuss the error between point values of $\beta (x)$ and $\beta_{\mathbf{X}}^\lam (x) $ for any $x \in \bX$. In view of the reproducing property of $K$ we have 
\begin{align}
	|\beta (x) - \beta^\lam_{\mathbf{X}} (x) | = \big|  \innerpro{K_{x},\beta - \beta^\lam_{\mathbf{X}} }_\hk \big| 
	& = \big|  \innerpro{K(\cdot , x),\beta - \beta^\lam_{\mathbf{X}} }_\hk \big|\nonumber\\
	&= \big| \innerpro{\xi(J^*_p J_p)v_x, \beta - \beta^\lam_{\mathbf{X}}}_\hk \big| \nonumber\\
	&\leq c \norm{\xi (J^*_p J_p) (\beta - \beta^\lam_{\mathbf{X}})  }_\hk. \label{eq:pw_beta}
\end{align}
Similarly, we obtain 
\begin{align*}
	|\beta (x) - \bar{\beta}^\lam (x) | &\leq c \norm{\xi (J^*_p J_p) (\beta - \bar{\beta}^\lam)}_\hk, \\
	|\bar{\beta}^\lam (x) - \beta^\lam_{\mathbf{X}} (x) | &\leq c \norm{\xi (J^*_p J_p) (\bar{\beta}^\lam - \beta^\lam_{\mathbf{X}}) }_\hk,
\end{align*}
that allows for the following decomposition of the error bound
\begin{align}\label{eq:pw_decom_bound}
		|\beta (x) - \beta^\lam_{\mathbf{X}} (x) | \leq c \left(\norm{\xi (J^*_p J_p) (\beta - \bar{\beta}^\lam)}_\hk + \norm{\xi (J^*_p J_p)(\bar{\beta}^\lam - \beta^\lam_{\mathbf{X}}) }_\hk \right).
\end{align}

In the following propositions, we estimate the terms on the right-hand side of (\ref{eq:pw_decom_bound}). 
\begin{proposition}\label{prop:pw_appro_bound} Let Assumption \ref{assum:source_cond_kernel} be satisfied. Assume also that $\beta$ and the regularization indexed by $g_\lam$ meet the conditions of Proposition \ref{prop:appro_bound}. Then 
		\begin{align*}
		\norm{\xi(J_p^*J_p) (\beta - \bar{\beta}^\lam )}_\hk \leq c  \xi(\lam)  \left( \Xi^{\frac32}\varphi(\lam) +  ( \gamma_0 + \gamma_\frac12 ) \Xi^{\frac12} \psi(\kappa_0) \norm{J^*_p J_p - \sxp^* \sxp }_{\hk \rightarrow \hk}  \right).
	\end{align*} 
for $\varphi \in \mathcal{F} (0,c)$, while for an operator monotone index function $\varphi$ we have 
	\begin{align*}
	\norm{\xi (J^*_p J_p) (\beta - \bar{\beta}^\lam)}_\hk \leq c (\gamma_0 + \gamma_{\frac{3}{2}}) \Xi^{\frac32} \xi(\lam) \varphi(\lam).
\end{align*}
\end{proposition}
\begin{proof}
	The analysis below is based on a modification of arguments developed in \cite{lu2020} for estimating the $L_{2,p}$-norm of any function $f \in L_{2,p}$ in terms of $\norm{(J_p^* J_p)^{\frac12} f}_\hk$. For the reader's convenience, we present this modification in detail.
	
	First of all, directly from Lemma A.1 \cite{lu2020}, it follows that, if $g_\lambda$ is any regularization with qualification $1$, then  
	\begin{align}\label{eq:A1.2}
		\norm{r_\lam (\sxp^*\sxp ) (\lam I + \sxp^*\sxp)^\frac12}_{\hk \to \hk} \leq  \sqrt{(\gamma_0^2 + \gamma_{\frac{1}{2}}^2) } \lam^\frac12.
	\end{align}	
	If $g_\lam$ has qualification at least $3/2$ then
	\begin{align}\label{eq:A1.3}
		\norm{r_\lam (\sxp^*\sxp ) (\lam I + \sxp^*\sxp)^\frac32}_{\hk \to \hk} \leq \sqrt{8(\gamma_0^2 + \gamma_{\frac{3}{2}}^2) } \lam^\frac32.
	\end{align}

	If $\beta$ meets source condition (\ref{eq:smooth3.2.1}), then for any $\lam > 0$ 
	\begin{align}\label{eq:1st_proof_app_bound}
	&\norm{\xi (J^*_p J_p) (\beta - \bar{\beta}^\lam)}_\hk = \norm{\xi(J^*_pJ_p) (I - g_\lam (\sxp^* \sxp ) \sxp^* \sxp ) \beta  }_\hk \nonumber \\
	&\hspace*{1.2cm}= \norm{\xi(J^*_pJ_p) (I - g_\lam (x^* \sxp ) \sxp^* \sxp ) \varphi(J^*_p J_p) v }_\hk \nonumber \\
	&\hspace*{1.2cm}\leq \norm{\xi (J_p^*J_p) (\lam I + J_p^*J_p)^{-\frac{1}{2}}}_{\hk \rightarrow \hk} \norm{(\lam I + J_p^*J_p)^{\frac{1}{2}} r_\lam (\sxp^* \sxp )\varphi(J_p^*J_p) v}_{\hk}.
	\end{align}
	Now we are going to estimate each component on the right-hand side of (\ref{eq:1st_proof_app_bound}). From the proof of Lemma \ref{lem:N_infty_bound}, we have 
	\begin{align}\label{eq:1st_proof_app_bound_1}
		\norm{\xi (J_p^*J_p) (\lam I + J_p^*J_p)^{-\frac{1}{2}}}_{\hk \rightarrow \hk} \leq c \frac{\xi(\lam)}{\sqrt{\lam}}.
	\end{align}
	Observe also that 
	\begin{align*}
		&\norm{(\lam I + J_p^*J_p)^{\frac{1}{2}} r_\lam (\sxp^* \sxp )\varphi(J_p^*J_p) v}_{\hk} \nonumber\\
		&\hspace*{1.7cm}\leq \norm{(\lam I + J_p^*J_p)^{\frac{1}{2}} r_\lam (\sxp^* \sxp ) (\lam I + J_p^*J_p)}_{\hk \to \hk} \norm{(\lam I+ J_p^*J_p )^{-1}\varphi(J_p^*J_p) v}_{\hk}.
	\end{align*}
	Moreover, using (\ref{def:Xi}) and the bounds (\ref{eq:1/2Xi}) and (\ref{eq:A1.3}), we get
	\begin{align*}
		&\norm{(\lam I + J_p^*J_p)^{\frac{1}{2}} r_\lam (\sxp^* \sxp ) (\lam I + J_p^*J_p)}_{\hk \to \hk} \\ 
		&\hspace*{0.7cm} \leq  \norm{ (\lam I + J_p^*J_p)^{\frac{1}{2}} (\lam I + \sxp^* \sxp)^{-\frac{1}{2}} (\lam I + \sxp^* \sxp)^{\frac{1}{2}} r_\lam (\sxp^* \sxp ) (\lam I + J_p^*J_p)}_{\hk \to \hk} \\
		&\hspace*{0.7cm} \leq \Xi^{\frac12} \norm{ (\lam I + \sxp^* \sxp)^{\frac{1}{2}} r_\lam (\sxp^* \sxp ) (\lam I + \sxp^* \sxp) (\lam I + \sxp^* \sxp)^{-1} (\lam I + J_p^*J_p) }_{\hk \to \hk} \\
		&\hspace*{0.7cm} \leq \Xi^{\frac12} \norm{(\lam I + \sxp^* \sxp)^{\frac{1}{2}} r_\lam (\sxp^* \sxp ) (\lam I + \sxp^* \sxp) }_{\hk \to \hk} \Xi\\
		&\hspace*{0.7cm} = \norm{r_\lam (\sxp^* \sxp ) (\lam I + \sxp^* \sxp)^{\frac{3}{2}}}_{\hk \to \hk} \Xi^{\frac32} \\
		&\hspace*{0.7cm}  \leq c \Xi^{\frac32} \lam^\frac32  \left( \gamma_0 +    \gamma_\frac32  \right).
	\end{align*}
	Besides, using the same argument as in the proof of Lemma \ref{lem:N_infty_bound}, for an operator monotone index function $\varphi$ we have
	\begin{align}\label{ineq_varphi_v}
		\norm{(\lam I+ J_p^*J_p )^{-1}\varphi(J_p^*J_p) v}_{\hk} \leq c \frac{\varphi(\lam)}{\lam} \norm{v}_\hk.
	\end{align}
	Thus,
	\[
	\norm{(\lam I + J_p^*J_p)^{\frac{1}{2}} r_\lam (\sxp^* \sxp )\varphi(J_p^*J_p) v}_{\hk} \leq c \left( \gamma_0 +    \gamma_\frac32  \right) \Xi^{\frac32}  \lam^\frac12  \varphi (\lam).
	\]
	Substituting (\ref{eq:1st_proof_app_bound_1}) and the above estimate into (\ref{eq:1st_proof_app_bound}), we obtain the second bound of the proposition.
%	\begin{align*}
%		\norm{\xi (J^*_p J_p) (\beta - \bar{\beta}^\lam)}_\hk \leq c (\gamma_0 + \gamma_{\frac{3}{2}}) \Xi^{\frac32} \xi(\lam) \varphi(\lam).
%	\end{align*}

	Now we turn to proving the first bound and assume that $\beta$ meets (\ref{eq:smooth3.2.1}) with $\varphi = \vartheta \psi$. Then we have
%	If $\beta$ meets source condition (\ref{eq:smooth3.2.1}), where $\varphi \in \mathcal{F} (0,c)$ with $c$ is large enough and the regularization indexed by $g_\lam $ has a qualification $\vartheta(t) t^\frac32$. We take the following decomposition
%	\begin{align*}
%		\beta - \bar{\beta}^\lam = & (I - g_\lam(\sxp^* \sxp) \sxp^* \sxp ) \times \\
%		& \times \left[ \vartheta(\sxp^* \sxp) \psi(J^*_p J_p) + (\vartheta(J_p^*J_p) - \vartheta(\sxp^* \sxp))\psi(J_p^*J_p) \right]v.
%	\end{align*}
	\begin{align}\label{1st_bound_of_prop5.1}
		\norm{\xi(J_p^*J_p) (\beta - \bar{\beta}^\lam )}_\hk \leq & \norm{ \xi(J^*_p J_p)  r_\lam (\sxp^* \sxp)\vartheta(\sxp^* \sxp) \psi (J^*_p J_p) v}_\hk \nonumber \\
		& +\norm{\xi(J_p^*J_p) r_\lam (\sxp^* \sxp) (\vartheta(J^*_p J_p) -\vartheta(\sxp^* \sxp) ) \psi(J^*_p J_p) v }_\hk.
	\end{align}
	Further, we estimate separately each term on the right-hand side of (\ref{1st_bound_of_prop5.1}). By using (\ref{eq:1st_proof_app_bound_1}) and (\ref{ineq_varphi_v}), the first term is estimated as follows:
	\begin{align} \label{term_1_of_1st_bound_prop5.1}
		&\norm{ \xi(J^*_p J_p)  r_\lam (\sxp^* \sxp)\vartheta(\sxp^* \sxp) \psi (J^*_p J_p) v}_\hk \nonumber \\
		& \hspace*{0.7cm} \leq \norm{\xi(J^*_p J_p) (\lam I+ J^*_p J_p )^{-\frac{1}{2}}}_{\hk \rightarrow \hk} \norm{(\lam I+ J^*_p J_p )^{\frac{1}{2}}r_\lam (\sxp^* \sxp)\vartheta(\sxp^* \sxp) \psi (J^*_p J_p) v}_\hk \nonumber \\
		& \hspace*{1.4cm} \times \norm{(\lam I+ J^*_p J_p )^{-1}\psi (J^*_p J_p) v}_{\hk} \nonumber\\
		& \hspace*{0.7cm} \leq c \frac{\xi(\lam)}{\sqrt{\lam}} \norm{(\lam I+ J^*_p J_p )^{\frac{1}{2}} r_\lam (\sxp^* \sxp)\vartheta(\sxp^* \sxp) (\lam I+ J^*_p J_p )}_{\hk \rightarrow \hk} \frac{\psi(\lam)}{\lam} \norm{v}_\hk \nonumber \\
%		& \hspace*{0.7cm} \leq c \frac{\xi(\lam)\psi(\lam)}{{\lam}^\frac32} \norm{r_\lam (\sxp^* \sxp)\vartheta(\sxp^* \sxp) (\lam I+ \sxp^* \sxp)^{\frac{3}{2}} (\lam I+ \sxp^* \sxp )^{-\frac{3}{2}} (\lam I+ J^*_p J_p )^{\frac{3}{2}}  }_{\hk \to \hk}   \\
		& \hspace*{0.7cm} \leq c \frac{\xi(\lam)\psi(\lam)}{{\lam}^\frac32} \Xi^{\frac12} \norm{ (\lam I + \sxp^* \sxp)^{\frac{1}{2}} r_\lam (\sxp^* \sxp)\vartheta(\sxp^* \sxp)(\lam I + \sxp^* \sxp) }_{\hk \to \hk} \Xi \nonumber\\
		&\hspace*{0.7cm} \leq c \frac{\xi(\lam)\psi(\lam)}{{\lam}^\frac32}  \Xi^{\frac32} \vartheta (\lam) \lam^\frac32 \nonumber\\
		&\hspace*{0.7cm} \leq c \xi(\lam)\varphi(\lam) \Xi^{\frac32},
	\end{align}
	and with the use of (\ref{eq:A1.2}) we can estimate the second term in (\ref{1st_bound_of_prop5.1}) as
	\begin{align}\label{term_2_of_1st_bound_prop5.1}
		&\norm{\xi(J_p^*J_p) r_\lam (\sxp^* \sxp) (\vartheta(J^*_p J_p) -\vartheta(\sxp^* \sxp) ) \psi(J^*_p J_p) v }_\hk \nonumber\\
		& \hspace*{0.5cm} \leq \norm{\xi (J_p^*J_p) (\lam I + J_p^*J_p)^{-\frac{1}{2}} (\lam I + J_p^*J_p)^{\frac{1}{2}} r_\lam (\sxp^* \sxp) (\vartheta(J^*_p J_p) -\vartheta(\sxp^* \sxp) ) \psi(J^*_p J_p) v }_\hk \nonumber\\
		& \hspace*{0.5cm} \leq c \frac{\xi(\lam)}{\sqrt{\lam}} \norm{ (\lam I + J_p^*J_p)^{\frac{1}{2}} r_\lam (\sxp^* \sxp)}_{\hk \to \hk}  \norm{(\vartheta(J^*_p J_p) -\vartheta(\sxp^* \sxp) )  }_{\hk \to \hk} \nonumber\\
		& \hspace*{0.5cm} \leq c \frac{\xi(\lam)}{\sqrt{\lam}} \norm{ (\lam I + J_p^*J_p)^{\frac{1}{2}} r_\lam (\sxp^* \sxp) }_{\hk \rightarrow \hk} \norm{J^*_p J_p - \sxp^* \sxp }_{\hk \rightarrow \hk} \psi(\kappa_0) \norm{v}_\hk \nonumber\\
		& \hspace*{0.5cm} \leq c \frac{\xi(\lam)}{\sqrt{\lam}} \Xi^{\frac12} \norm{(\lam I + \sxp^* \sxp)^{\frac{1}{2}} r_\lam (\sxp^* \sxp) }_{\hk \rightarrow \hk} \norm{J^*_p J_p - \sxp^* \sxp }_{\hk \rightarrow \hk} \nonumber\\
		& \hspace*{0.5cm} \leq c \xi(\lam) \Xi^{\frac12} \left( \gamma_0 + \gamma_\frac12 \right) \norm{J^*_p J_p - \sxp^* \sxp }_{\hk \rightarrow \hk}.
	\end{align}
	Substituting (\ref{term_1_of_1st_bound_prop5.1}) and (\ref{term_2_of_1st_bound_prop5.1}) into (\ref{1st_bound_of_prop5.1}), we obtain 
	\begin{align*}
		\norm{\xi(J_p^*J_p) (\beta - \bar{\beta}^\lam )}_\hk \leq c  \xi(\lam)  \left( \Xi^{\frac32}\varphi(\lam) +  ( \gamma_0 + \gamma_\frac12 ) \Xi^{\frac12} \norm{J^*_p J_p - \sxp^* \sxp }_{\hk \rightarrow \hk}  \right).
	\end{align*}
\end{proof}

\begin{proposition}\label{prop:pw_propa_bound}
	Assume that Assumption \ref{assum:source_cond_kernel} be satisfied. Then it holds
	$$\norm{\xi(J_p^*J_p) (\bar{\beta}^\lam - \beta_{\mathbf{X}}^\lam)}_\hk \leq c \frac{\xi(\lam)}{\sqrt{\lam}} \Xi (\gamma_{-1} + \gamma_0 + 1) \norm{(\lam I + J_p^*J_p)^{-\frac12} (\sxq^* \sxq \mathbf{1} - \sxp^* \sxp \beta )}_\hk. $$
\end{proposition}

\begin{proof}
	Using (\ref{eq:smooth3.2.1}), (\ref{def:Xi}) and (\ref{eq:1st_proof_app_bound_1}), we derive
	\begin{align*}
		&\norm{\xi(J_p^*J_p) (\bar{\beta}^\lam - \beta_{\mathbf{X}}^\lam)}_\hk \leq  \norm{\xi(J_p^*J_p) g_\lam (\sxp^* \sxp) (\sxq^* \sxq \mathbf{1} - \sxp^* \sxp \beta )}_\hk \\
		&\hspace*{1.5cm}\leq \norm{\xi(J_p^*J_p) (\lam I + J_p^*J_p)^{-\frac12}}_{\hk \rightarrow \hk} \times \\
		&\hspace*{5.4cm} \times \norm{(\lam I + J_p^*J_p)^{\frac12} g_\lam (\sxp^* \sxp) (\sxq^* \sxq \mathbf{1} - \sxp^* \sxp \beta )}_\hk\\
%		&\hspace*{1cm}\leq c \frac{\xi(\lam)}{\sqrt{\lam}} \norm{(\lam I + J_p^*J_p)^{\frac12} g_\lam (\sxp^* \sxp) (\lam I + J_p^*J_p)^{\frac12} (\lam I + J_p^*J_p)^{-\frac12} (\sxq^* \sxq \mathbf{1} - \sxp^* \sxp \beta )}_\hk\\
		&\hspace*{1.5cm}\leq c \frac{\xi(\lam)}{\sqrt{\lam}} \norm{(\lam I + J_p^*J_p)^{\frac12} g_\lam (\sxp^* \sxp) (\lam I + J_p^*J_p)^{\frac12} }_{\hk \to \hk }\times \\
		&\hspace*{5.8cm} \times  \norm{(\lam I + J_p^*J_p)^{-\frac12} (\sxq^* \sxq \mathbf{1} - \sxp^* \sxp \beta )}_\hk\\
		&\hspace*{1.5cm}\leq c \frac{\xi(\lam)}{\sqrt{\lam}} \Xi^\frac12  \norm{ g_\lam (\sxp^* \sxp) (\lam I +\sxp^* \sxp) }_\hk \Xi^\frac12 \times \\
		&\hspace*{5.8cm} \times \norm{(\lam I + J_p^*J_p)^{-\frac12} (\sxq^* \sxq \mathbf{1} - \sxp^* \sxp \beta )}_\hk \\
		&\hspace*{1.5cm}\leq c \frac{\xi(\lam)}{\sqrt{\lam}} \Xi \sup_t \big| g_\lam (t) (\lam + t) \big| \norm{(\lam I + J_p^*J_p)^{-\frac12} (\sxq^* \sxq \mathbf{1} - \sxp^* \sxp \beta )}_\hk \\
		&\hspace*{1.5cm}\leq c \frac{\xi(\lam)}{\sqrt{\lam}} \Xi (\gamma_{-1} + \gamma_0 + 1) \norm{(\lam I + J_p^*J_p)^{-\frac12} (\sxq^* \sxq \mathbf{1} - \sxp^* \sxp \beta )}_\hk.
	\end{align*}
\end{proof}

Now we can combine (\ref{eq:pw_decom_bound}) with Propositions \ref{prop:pw_appro_bound}, \ref{prop:pw_propa_bound} and with Lemma \ref{lem1}. Then the same argument as in the proof of Proposition \ref{prop:combine_app_propa_bound} gives us the following statement.

\begin{theorem}\label{thrm5.3}
		Under the assumption of Propositions \ref{prop:pw_appro_bound} and \ref{prop:pw_propa_bound}, for $\lambda > \lam_*$ with probability at least $1 - \delta$, for all $x \in \bX$, we have 
			\begin{align*}
			|\beta (x) - \beta^\lam_{\mathbf{X}} (x) |  \leq c \xi (\lam) \left( \varphi(\lam) +  (m^{-\frac12} + n^{-\frac12}) \frac{\xi (\lam)}{\lam} \right) \left( \log \frac{2}{\delta}\right)^2,
		\end{align*}
	and for $\lam = \lambda_{m,n} = \theta_{\varphi,\xi}^{-1} (m^{-\frac12} + n^{-\frac12}),$
	\[ |\beta (x) - \beta^\lam_{\mathbf{X}} (x) |  \leq c \xi (\lam_{m,n}) \varphi(\lam_{m,n}) \log^2 \frac{1}{\delta}.\]
\end{theorem}

\begin{remark}\label{rem5.4}
%	As has been emphasized in Remark \ref{rem_sec4}, the accuracy of (\ref{eq:beta_g_lam}) has been analyzed in RKHS with the best possible rate $O \left((m^{-\frac12} + n^{-\frac12})^\frac{\eta }{\eta +1 - \varsigma } \right)$. 
 
 Let us consider the same index functions $\varphi (t) = t^\eta$ and $\xi (t)= t^\varsigma$ as in Remark \ref{rem_sec4}, where the accuracy of order $O \left((m^{-\frac12} + n^{-\frac12})^\frac{\eta }{\eta +1 - \varsigma } \right)$ has been derived for (\ref{eq:beta_g_lam}). Under the same assumptions, Theorem \ref{thrm5.3} guarantees the accuracy of order $O \left((m^{-\frac12} + n^{-\frac12})^\frac{\eta +\varsigma}{\eta +1 - \varsigma } \right)$.  This illustrates that the reconstruction of the Radon-Nikodym derivative at any particular point can be done with much higher accuracy than its reconstruction as an element of RKHS. But let us stress that the above high order of accuracy is guaranteed when the qualification $s$ of the used regularization scheme is higher than that of the Lavrentiev regularization or KuLSIF (\ref{eq:beta_tilde}).
 
 %the bound in Theorem \ref{thrm5.3} 
 %gives the rate of order $O \left((m^{-\frac12} + n^{-\frac12})^\frac{\eta +\varsigma}{\eta +1 - \varsigma } \right)$. Consequently, a pointwise approximation of $\beta$ can be constructed with a higher rate of accuracy.

\end{remark}

\section{Numerical illustrations}
%\subsection{Regularized Radon-Nikodym numerical differentiation}
%Our first three illustrations are on toy data, and they are intended mainly to demonstrate potential advantages of the use of the general regularization scheme (\ref{eq:beta_g_lam}) corresponding to Tikhonov regularization and indexed by $g_\lambda (t) = (\lambda + t)^{-1}.$

In our examples, we simulate inputs $X_p = (x_1,x_2,\ldots,x_n)$ to be sampled from the normal distribution $p \sim N(2, 5)$, while the inputs $X_q = (x_1',x_2',\ldots,x_m')$ are sampled from the normal distribution $q \sim N (\mu_q,0.5)$ with $\mu_q = \{2,3,4\}$. In this case, the Radon-Nikodym derivative $\beta = \frac{dq}{dp}$ is known to be $$\beta(x) =  \sqrt{10} e^{\frac{(x-2)^2 - 10(x - \mu_q)^2}{10}} .$$

%The outputs are simulated as values of the function $f^*(x) = 1 + e^{-\frac{x^2}{2}}$  observed in Gaussian white noise, such that $y_i = f^* (x_i) +\epsilon_i,$ $i = 1,2,\ldots,n$, where $\epsilon_i$ are zero-mean Gaussian random variables with standard deviation $\delta$. 

In the algorithms described in Section 3, we choose the kernel as 
\begin{align*}
	K(x, x' ) = 1 + e^{- \frac{(x-x')^2}{2}},
\end{align*}
which is a combination of a universal Gaussian kernel with a constant such that the corresponding space $\mathcal{H}_K$ contains all constant functions.

We are going to illustrate that to achieve high order of accuracy for reconstruction of the Radon-Nikodym derivative at any particular point, as it is guaranteed by Theorem \ref{thrm5.3}, one needs to employ a regularization with the qualification that is higher than 1. For doing this we use a particular case of the general regularization scheme (\ref{eq:beta_g_lam}), namely the iterated Lavrentiev regularization, to compute the values of the approximate Radon-Nikodym derivative $\beta_{\mathbf{X}}^\lam = (\beta_1^\lam,\beta_2^\lam,\cdots,\beta_n^\lam)$ with $\beta_i^\lambda = \beta_\bX^\lam (x_i)$. 

Recall that the $k$ times iterated Lavrentiev regularization is indexed by the functions

%As in Remark \ref{rem5.4}, we stress that to rely on the error bounds given by Theorem \ref{thrm5.3} one needs to employ regularization schemes, which have qualification $s>1$. In view of this, we use iterated Laventiev regularization, where
\begin{align}\label{eq6.1.1}
	g_\lambda (t) = g_{\lambda, k} (t) = \left( 1 - \frac{\lambda^k}{(\lambda + t)^k}\right) t^{-1},
\end{align}
and has the qualification $s=k$. 
%As we already mentioned in Section \ref{sec:gen_reg_scheme}, Tikhonov regularization corresponds to (\ref{eq:beta_g_lam}), (\ref{eq6.1.1}) with $k=1$, while for $k = 2,3,\ldots,$  (\ref{eq:beta_g_lam}), (\ref{eq6.1.1}) can be computed by applying iterated Tikhonov regularization $k$ times.

For $g_\lambda (t) = g_{\lambda, k} (t) $ the vector of values of the approximate Radon-Nikodym derivative $\beta^\lambda_{\mathbf{X}} = \beta^{\lambda,k}_{\mathbf{X}} $ given by (\ref{eq:beta_g_lam}), (\ref{eq6.1.1}) is the $k$-th term of the sequence 
\begin{align}\label{eq6.1.2}
		\beta^{\lambda,l}_{\mathbf{X},0} &= 0,\nonumber \\
	\beta^{\lambda,l}_{\mathbf{X}} &= \left( n \lambda \mathbf{I} + \mathbf{K} \right)^{-1} \left( n \lambda \beta^{\lambda,l-1}_{\mathbf{X}}  + \bar{F} \right),\hspace*{0.5cm} l = 1,2,\ldots,k.	
\end{align}
where $\mathbf{I}$ is $n$ by $n$ identity matrix, $\mathbf{K}  = \left( K(x_i,x_j)\right)_{i,j = 1}^n$, and $\bar{F} = (F_i)_{i=1}^n$ with $ F_i = \frac{n}{m} \sum_{j=1}^m K(x_i,x_j').$

The algorithm (\ref{eq6.1.2}) has been implemented with $ m = n = 100$ and $k = \{1,2,3,5,10\}$. The regularization parameter $\lambda$ is chosen by the so-called quasi-optimality criterion (see, for example, \cite{bauer_2008}, \cite{Kindermann_2018}), $\bar{\lam} \in \{\lam_\iota = \lambda_0 \varrho^\iota, \iota = 1,2,\ldots,w \}, \varrho < 1$ such that for $\bar{\lam} = \lam_{\iota_0},$ 
$$\norm{\beta^{\lam_{\iota_0}}_\bX - \beta^{\lam_{\iota_0 - 1}}_\bX}_{\R^n} = \min \left\{  \norm{\beta^{\lam_{\iota}}_\bX - \beta^{\lam_{\iota - 1}}_\bX}_{\R^n}, \iota = 1,2,\ldots,w \right\}.$$
Taking into consideration Theorem \ref{thrm5.3} and Figure \ref{fig_lam}, one can expect that $\bar{\lambda} \approx \lam_{m,n} > (m^{-\frac{1}{2}}+n^{-\frac{1}{2}})$. Therefore, for $n=m=100$ it is natural to look for $\bar{\lambda}$ within interval $[0.1,0.9]$, and in our experiments we choose $\lambda_0 = 0.9$, $\varrho = \sqrt[9]{\frac19}$, and $w = 9$, such that $\lam_{\iota} \in [0.1,0.9]$.
\begin{figure}[!ht]%
	\centering
	\subfloat[]{\includegraphics[scale=0.135]{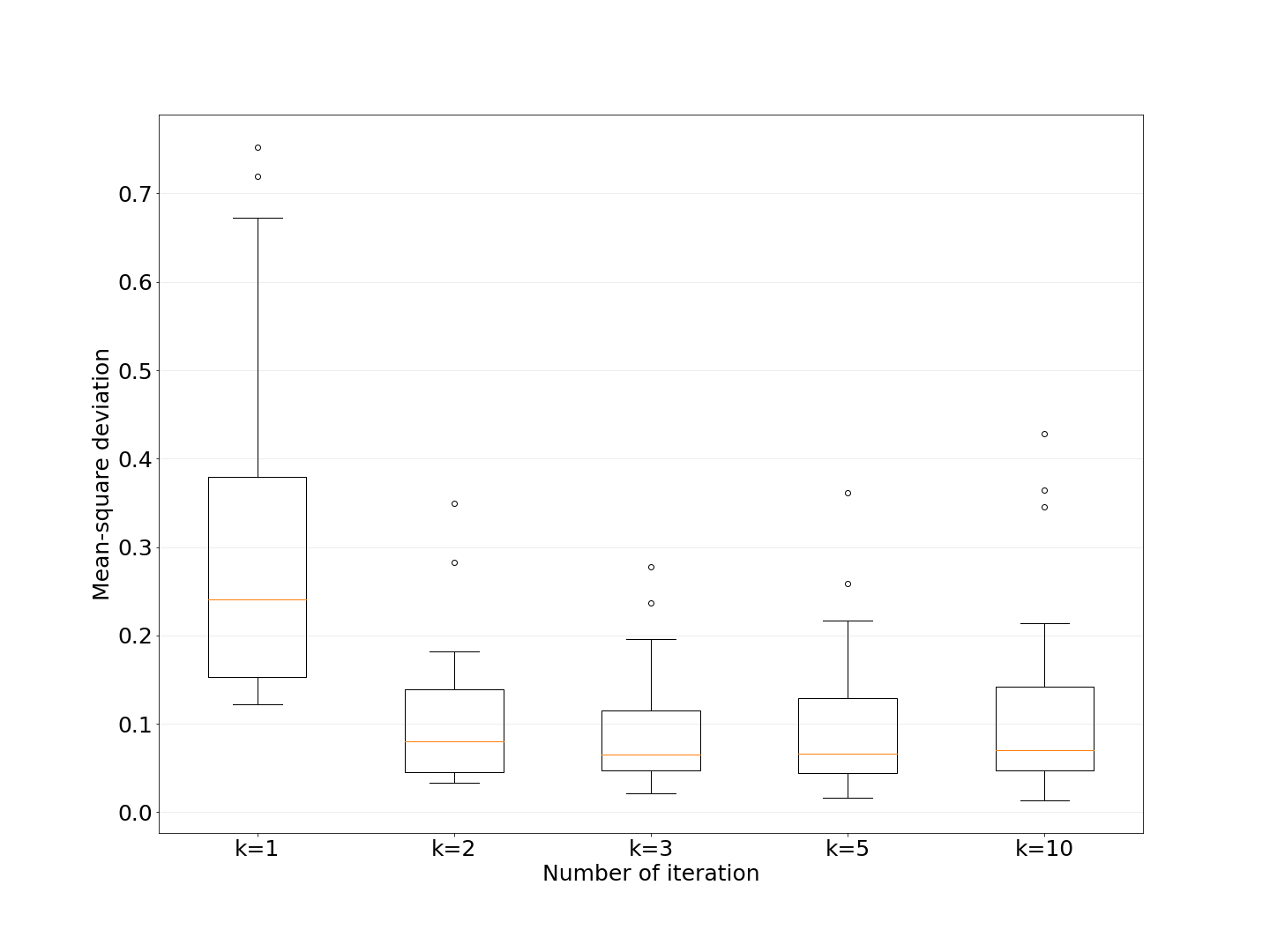}\label{fig1:a}}%
	\subfloat[]{\includegraphics[scale=0.135]{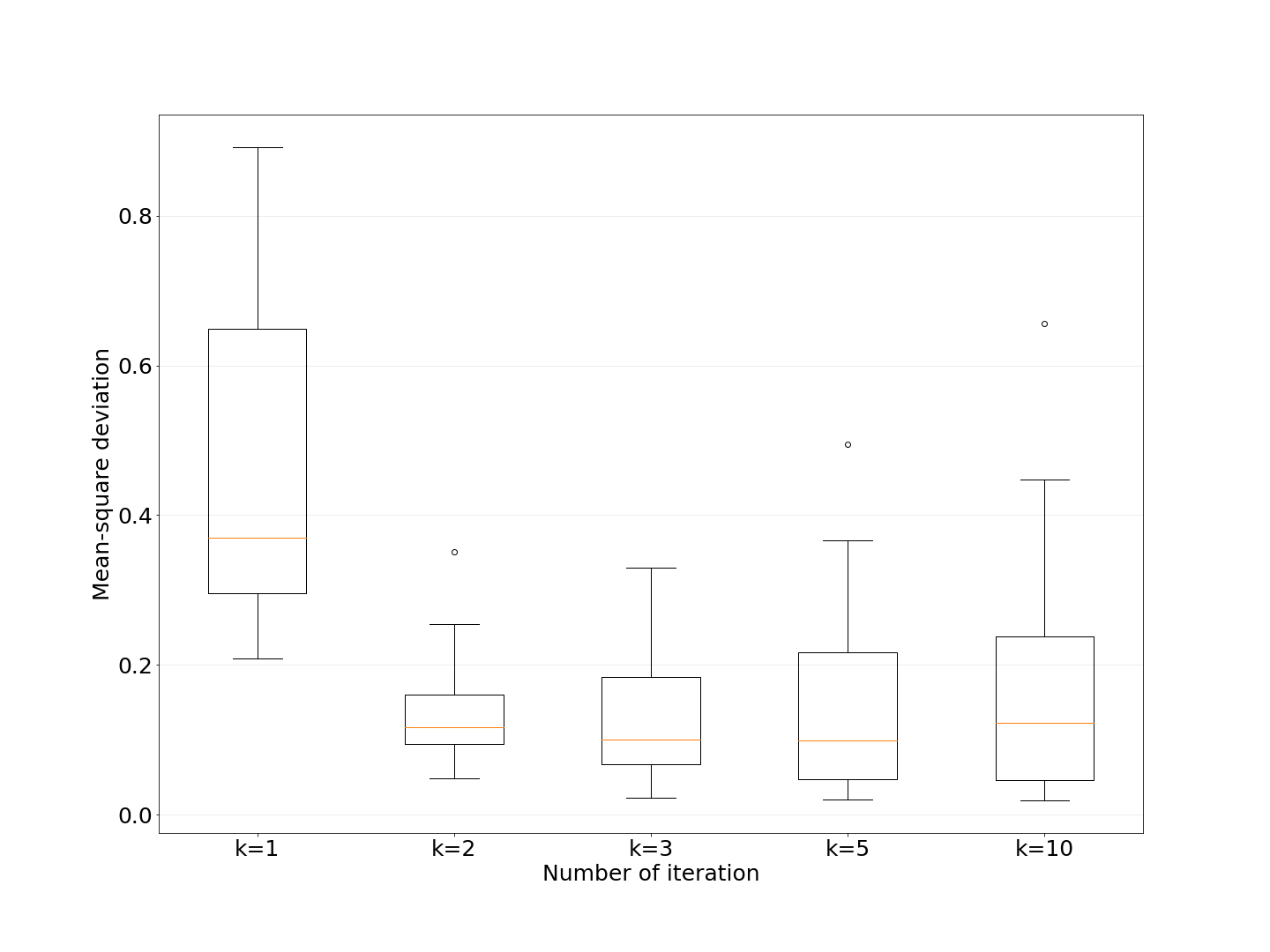}\label{fig1:b}}\\
	\subfloat[]{\includegraphics[scale=0.135]{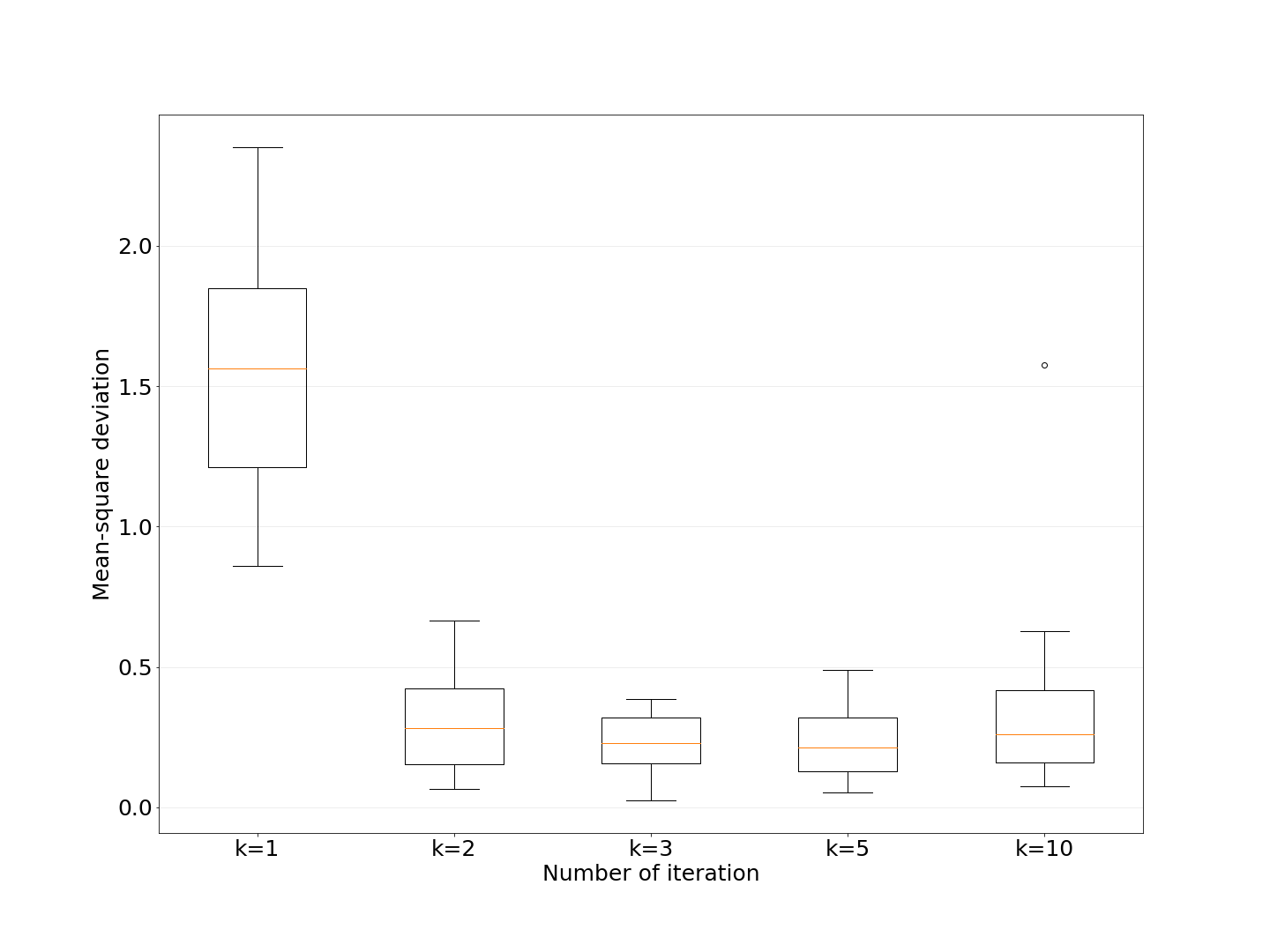}\label{fig1:c}}%
	\caption{Mean-square deviation in examples with (a) $X_q \sim N(2,0.5)$, (b) $X_q \sim N(3,0.5)$, and (c) $X_q \sim N(4,0.5)$.}%
	\label{fig1}%
\end{figure}

The performance of each implementation has been measured in terms of the mean-square deviation (MSD).
$$MSD =  n^{-1} \sum_{i=1}^n  \left( \beta (x_i) - \beta^{\lambda,k}_{\mathbf{X}} (x_i)\right)^2.$$

A summary of the performance over 20 simulations of $(x_i)_{i=1}^n$, $(x_j')_{j=1}^m$ in all cases $\mu_q = \{2,3,4\}$ is presented in the form of box plots in Figure \ref{fig1}. It can be clear seen that in our examples the considered realization of the iterated Laventiev regularization outperforms its original version $(k=1)$. This supports a conclusion from Theorem \ref{thrm5.3} suggesting the use of high qualification regularization for pointwise evaluation of Radon-Nikodym derivation.

The performance of the algorithm (\ref{eq6.1.2}) for a particular simulation is displayed in Figure \ref{fig2}. In this figure, the exact values $\beta$ are shown by the line, and the $\beta^{\lambda,1}_{\mathbf{X}} (x_i)$, $\beta^{\lambda,2}_{\mathbf{X}} (x_i)$, $\beta^{\lambda,3}_{\mathbf{X}} (x_i)$, $\beta^{\lambda,5}_{\mathbf{X}} (x_i)$, and $\beta^{\lambda,10}_{\mathbf{X}} (x_i)$ are denoted correspondingly by green triangles, red squares, cyan diamonds, yellow stars, and blue crosses.

\begin{figure}[!h]%
	\centering
	\subfloat[]{\includegraphics[scale=0.18]{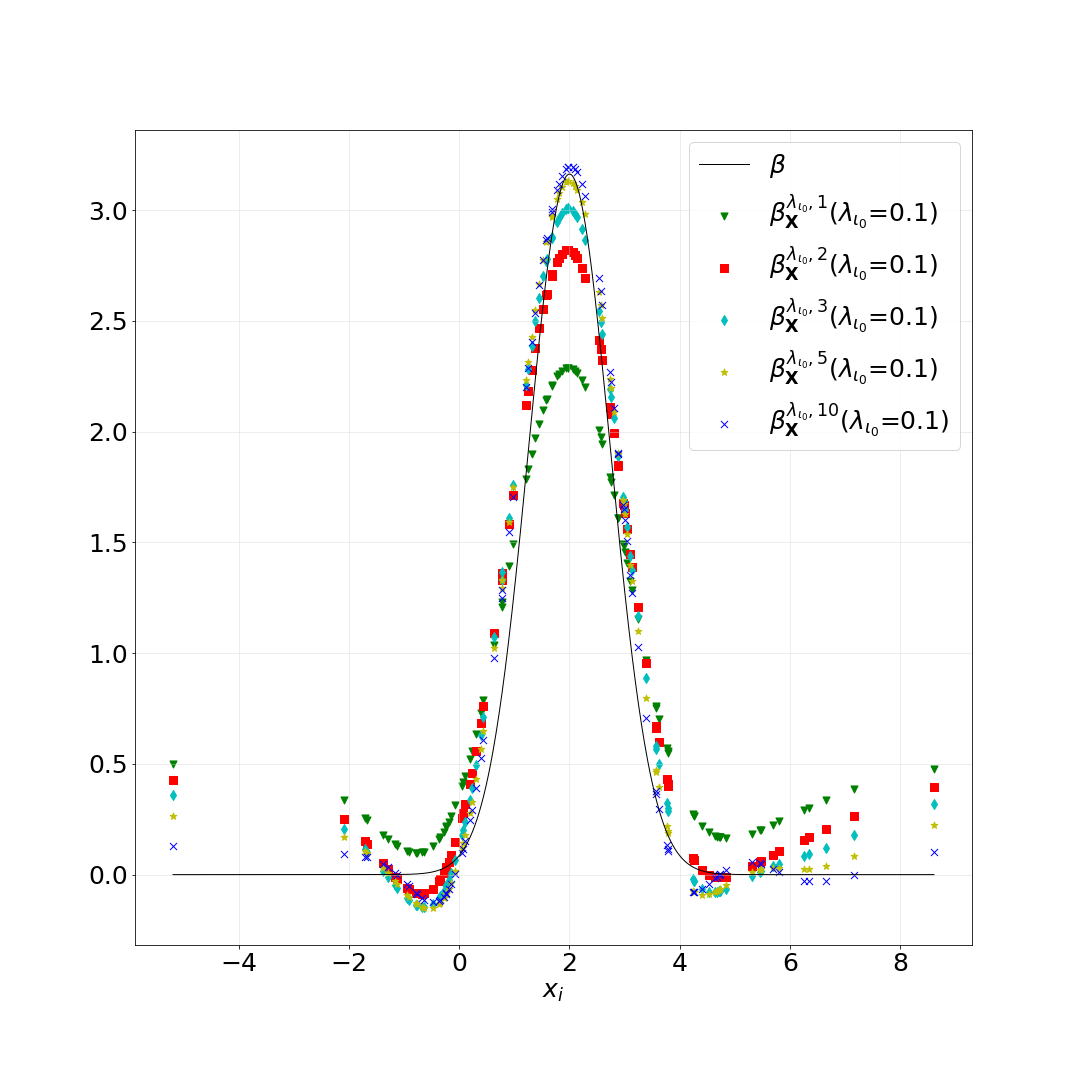}\label{fig2:a}}%
	\subfloat[]{\includegraphics[scale=0.18]{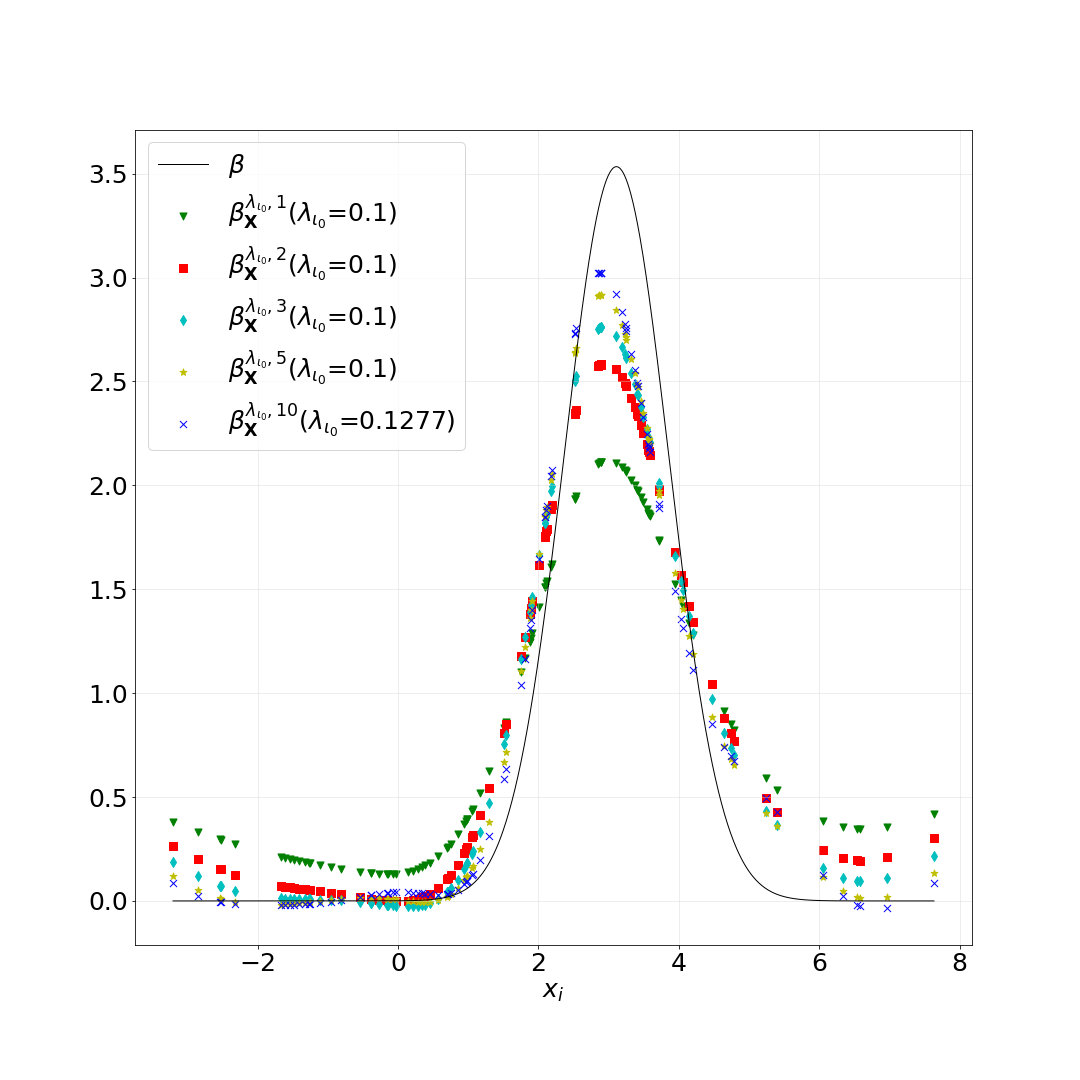}\label{fig2:b}}\\
	\subfloat[]{\includegraphics[scale=0.18]{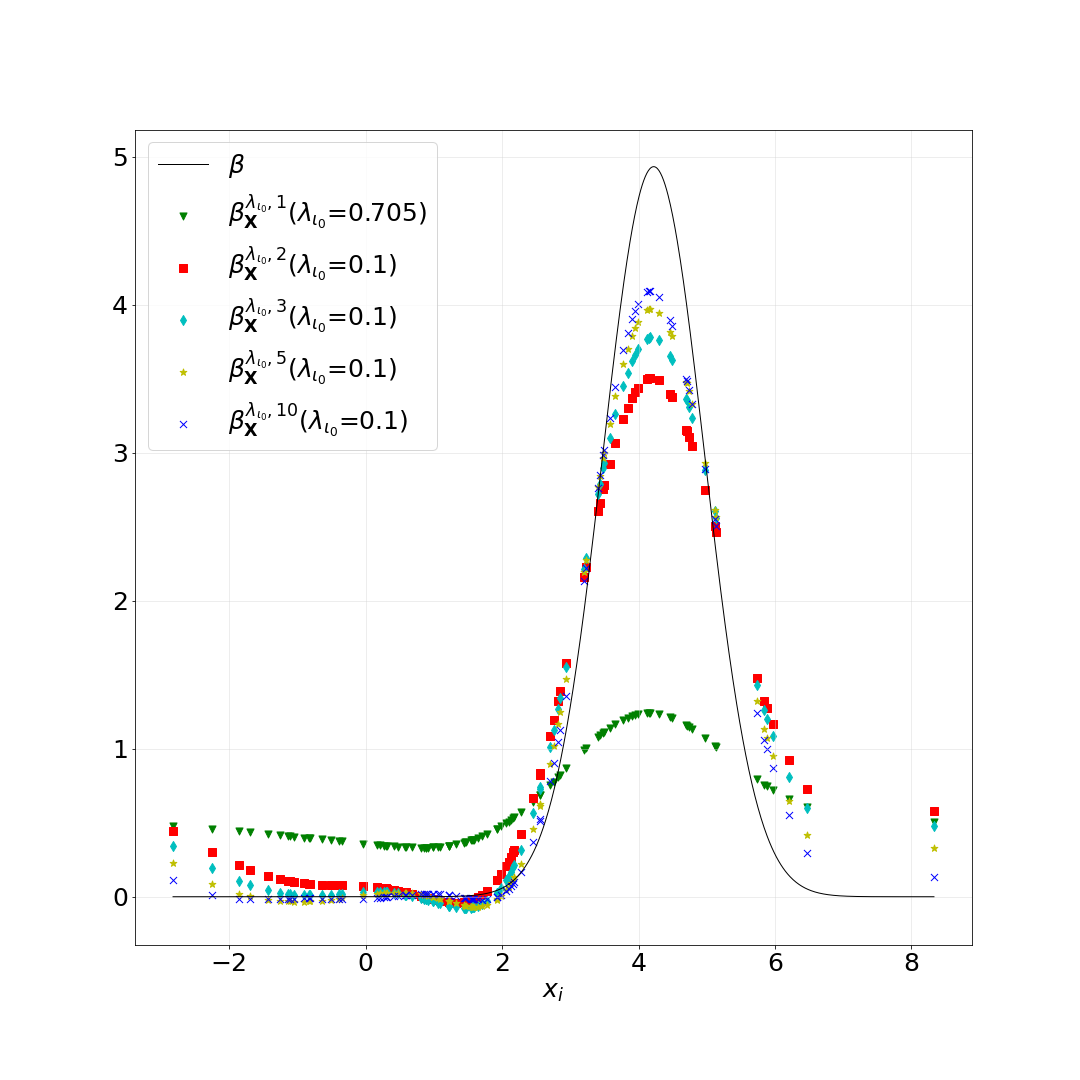}\label{fig2:c}}%
	\caption{The performance of the algorithm for a particular simulation with (a) $X_q \sim N(2,0.5)$, (b) $X_q \sim N(3,0.5)$, and (c) $X_q \sim N(4,0.5)$.}%
	\label{fig2}%
\end{figure}

% Acknowledgements and Disclosure of Funding should go at the end, before appendices and references

\acks{The research reported in this paper has been funded by the Federal Ministry for Climate Action, Environment, Energy, Mobility, Innovation and Technology (BMK), the Federal Ministry for Digital and Economic Affairs (BMDW), and the Province of Upper Austria in the frame of the COMET--Competence Centers for Excellent Technologies Programme and the COMET Module S3AI managed by the Austrian Research Promotion Agency FFG.}


%\newpage

%\appendix
%\section{}
%\label{app:theorem}

% Note: in this sample, the section number is hard-coded in. Following
% proper LaTeX conventions, it should properly be coded as a reference:

%In this appendix we prove the following theorem from
%Section~\ref{sec:textree-generalization}:

%In this appendix we prove the following theorem from
%Section~6.2:

%\noindent
%{\bf Theorem} {\it Let $u,v,w$ be discrete variables such that $v, w$ do
%not co-occur with $u$ (i.e., $u\neq0\;\Rightarrow \;v=w=0$ in a given
%dataset $\dataset$). Let $N_{v0},N_{w0}$ be the number of data points for
%which $v=0, w=0$ respectively, and let $I_{uv},I_{uw}$ be the
%respective empirical mutual information values based on the sample
%$\dataset$. Then
%\[
%	N_{v0} \;>\; N_{w0}\;\;\Rightarrow\;\;I_{uv} \;\leq\;I_{uw}
%\]
%with equality only if $u$ is identically 0.} \hfill\BlackBox

%\section{}

%\noindent
%{\bf Proof}. We use the notation:
%\[
%P_v(i) \;=\;\frac{N_v^i}{N},\;\;\;i \neq 0;\;\;\;
%P_{v0}\;\equiv\;P_v(0)\; = \;1 - \sum_{i\neq 0}P_v(i).
%\]
%These values represent the (empirical) probabilities of $v$
%taking value $i\neq 0$ and 0 respectively.  Entropies will be denoted
%by $H$. We aim to show that $\fracpartial{I_{uv}}{P_{v0}} < 0$....\\

%{\noindent \em Remainder omitted in this sample. See http://www.jmlr.org/papers/ for full paper.}

\vskip 0.2in
\bibliography{references}

\end{document}